\documentclass[11pt]{amsart}
\usepackage{amsfonts}
\usepackage{amsmath}
\usepackage{amssymb}
\usepackage{multicol}
\usepackage{stmaryrd}
\usepackage{cite}
\usepackage{epsfig}
\usepackage{color}
\usepackage{graphics}
\usepackage{graphicx}
\usepackage{epstopdf}
\usepackage{multicol,graphics}
\newcommand\bes{\begin{eqnarray}}
\newcommand\ees{\end{eqnarray}}

\newtheorem{theorem}{Theorem}[section]
\newtheorem{lemma}[theorem]{Lemma}
\newtheorem{corollary}[theorem]{Corollary}

\newtheorem{remark}[theorem]{Remark}
\newtheorem{proposition}[theorem]{Proposition}
\numberwithin{equation}{section}
\allowdisplaybreaks

\begin{document}

\title[Spatial Propagation in Nonlocal Dispersal Fisher-KPP Equations]{\textbf{Spatial Propagation in Nonlocal Dispersal Fisher-KPP Equations}}

\author[W.-B. Xu, W.-T. Li and S. Ruan]{Wen-Bing Xu$^{1,2}$, Wan-Tong Li$^{1,*}$ and Shigui Ruan$^{3}$}
\thanks{\hspace{-.6cm}
$^1$School of Mathematics and Statistics, Lanzhou University, Lanzhou, Gansu, 730000, People's Republic of China.
\\
$^2$Academy of Mathematics and Systems Science, Chinese Academy of Sciences, Beijing 100190, P. R. China.
\\
$^3$Department of Mathematics, University of Miami, Coral Gables, FL 33146, USA.
\\
$^*${\sf Corresponding author} (wtli@lzu.edu.cn)}

\date{\today}

\begin{abstract}
In this paper we focus on  three problems about the spreading speeds of nonlocal dispersal Fisher-KPP equations.  First, we study the signs of spreading speeds and find that they are determined by the asymmetry level of the nonlocal dispersal  and  $f'(0)$, where $f$ is the reaction function.  This indicates that asymmetric dispersal can influence the spatial dynamics in three aspects: it can determine the spatial propagation directions of solutions, influence the stability of equilibrium states, and affect the monotone property of solutions. Second, we give an improved proof of the spreading speed result by constructing new lower solutions and using the new ``forward-backward spreading'' method. Third, we establish the relationship between spreading speed and exponentially decaying initial data.  Our result demonstrates that when dispersal is symmetric, spreading speed decreases along with the increase of the exponentially decaying rate. In addition, the results on the signs of spreading speeds are applied to two  special cases where we present  more details of the influence of asymmetric dispersal.

\vspace{1em}

\end{abstract}

\maketitle

\section{Introduction}
\noindent

In this paper, we study spatial propagation of  the following nonlocal dispersal Fisher-KPP equation
\begin{equation}\label{1.1}
\left\{
\begin{aligned}
&u_t(t,x)= k*u(t,x)-u(t,x) + f(u(t,x)),~~t>0,~x\in \mathbb R,\\
&u(0,x)=u_0(x),~x\in\mathbb R,
\end{aligned}
\right.
\end{equation}
where $u_0\in C(\mathbb R)$, $f \in C^1([0,1])$ and satisfies the Fisher-KPP type condition:
\begin{itemize}
\item[(H)] $f$ is  monstable, namely $f(0)=f(1)=0$ and $f(u)>0$ for $u\in(0,1)$, $f'(0)>0$  and $f(u)\leqslant f'(0)u$ for $u\in (0,1)$.
\end{itemize}
The nonlocal dispersal, represented by the following integral operator
\[
k*u(t,x)-u(t,x)=\int_{\mathbb R}k(x-y)u(t,y)dy-u(t,x),
\]
can describe the  movements of organisms between not only  adjacent   but also nonadjacent spatial locations (see, e.g. Berestycki et al. \cite{BCH1}, Kao et al. \cite{KLS2010}, Murray \cite{Mur1993} and Wang \cite{wang2002}). Here the kernel $k(\cdot)$ is a continuous and nonnegative function with  $\int_{\mathbb R}k(x)dx=1$. Moreover, we assume that
\begin{itemize}
\item[(K1)]   there is a constant $\lambda>0$ such that $\int_{\mathbb R}k(x)e^{\lambda |x|}dx<+\infty$;
\item[(K2)]  $k(x_1)>0$ and $k(x_2)>0$ for some constants $x_1\in\mathbb R^+$ and $x_2\in\mathbb R^-$.
\end{itemize}
Assumption (K1) is called the Mollison condition.
For classical results on traveling wave solutions of equation \eqref{1.1}, we refer to Schumacher \cite{Sch1980}, Bates et al. \cite{BFRW}, Chen \cite{Chen1997}, Chen and Guo \cite{CG2003}, Carr and Chmaj \cite{CC2004}, Coville, D\'{a}vila and Mart\'{\i}nez \cite{CDM2008}, Yagisita \cite{Yag2009}, and Sun et al. \cite{SLW2011}. Entire solutions of equation \eqref{1.1} were studied by Li et al. \cite{LSW} and Sun et al. \cite{SZLW2018}.

Spreading speed is an important concept that describes the phenomenon of spatial propagation in many biological and ecological problems, such as the spatial spread of infection diseases and the invasion of species. In 1975, Aronson~and~Weinberger~\cite{AW1975} studied spreading speed of the following reaction-diffusion equation
\begin{equation}\label{1.4}
\left\{
\begin{aligned}
&u_t=u_{xx} + f(u),~~t>0,~x\in \mathbb R,\\
&u(0,x)=u_0(x),~x\in\mathbb R.
\end{aligned}
\right.
\end{equation}
When $f$ is monstable and $f'(0)>0$, they showed that
if $u_0(\cdot)\not\equiv0$ and $0\leqslant u(x)\leqslant1$ for $x\in\mathbb R$, then   $u(t,x)$   satisfies
\begin{equation}\label{1.5}
\lim\limits_{t\rightarrow+\infty} u(t,x)=1 \;\; \text{for any}~x\in\mathbb R.
\end{equation}
Moreover, if $u_0(x)$ is compactly supported on $\mathbb R$, then there is a  constant $c^*>0$ such that
\[
\lim\limits_{t\rightarrow+\infty}u(t,x+ct)=\left\{
\begin{aligned}
&1,&&|c|<c^*,\\
&0,&&|c|>c^*
\end{aligned}
\right. \;\;
\text{for any}~x\in\mathbb R.
\]
The constant  $c^*$ is  called the asymptotic speed of spread (for short, {\it spreading speed}) of equation \eqref{1.4}.
For more results on  spreading speed theory, we refer to Kolmogorov et al. \cite{KPP1937}, Aronson and Weinberger \cite{AW1975,AW1978},  Liang and Zhao \cite{LZ2008,LZ2010}, Lui \cite{Lui1989}, Weinberger \cite{Wei1982}, Weinberger et al. \cite{WLL2002}, Yi and Zou \cite{YZ2015}, and the references cited therein.

For the nonlocal dispersal equation \eqref{1.1}, Lutscher et al. \cite{LPL2005} considered the spreading speed  and proved that there are two constants $c_r^*$ and $c_l^*$ such that
\begin{equation}\label{1.2}
\lim\limits_{t\rightarrow+\infty}u(t,x+ct)=\left\{
\begin{aligned}
&1,&&c_l^*<c<c_r^*,\\
&0,&&c<c_l^*~\text{or}~c>c_r^*
\end{aligned}
\right.
\text{ for any }~x\in\mathbb R,
\end{equation}
where
\begin{equation}\label{1.7}
c_l^*\triangleq \sup\limits_{\lambda\in\mathbb R^-}\Big\{\lambda^{-1}\Big[\int_{\mathbb R}k(x)e^{\lambda x}dx-1+f'(0)\Big]\Big\},
\end{equation}
\begin{equation}\label{1.8}
c_r^*\triangleq \inf\limits_{\lambda\in\mathbb R^+}\Big\{\lambda^{-1}\Big[\int_{\mathbb R}k(x)e^{\lambda x}dx-1+f'(0)\Big]\Big\}.
\end{equation}
The constants $c_l^*$ and $c_r^*$  are  called \textit{spreading speeds to the left} and \textit{to the right} of the nonlocal dispersal equation, respectively. Note that $c_r^*$ may not be equal  to $-c_l^*$  because of the asymmetry of $k$. Here the asymmetry of $k$ means that the probability that organisms move from point $x$ to point $x+y$ is not equal to that  from   $x$ to  $x-y$.
In addition,  Finkelshtein et al. \cite{FKT2015,FKT2018} extended this conclusion  into high dimensional space  $\mathbb R^d$, which is more complex because of the radial asymmetry of kernels. For more results about  spreading speeds of nonlocal dispersal equations, we refer to Rawal et al. \cite{RSZ2015}, Shen and Zhang \cite{SZ2010} and Zhang et al. \cite{ZLW2012}.

The aim of this paper is to study some new problems on  spreading speeds of nonlocal dispersal equations. The three main topics we cover are:  identifying the signs of spreading speeds,   improving the  proof of the  spreading speed result  and establishing the relationship between spreading speed and exponentially decaying initial data, which we describe in turn next.

First, we identify the signs of spreading speeds. In reaction-diffusion equation,  the spreading speed to the right $c^*$ is always positive and that to the left $-c^*$ is always negative. We wonder whether this  remains true in nonlocal dispersal equations. It is significant to  identify the signs of spreading speeds, since they have important influences on   spatial property of solutions and   stability of equilibrium states (see the influences on spatial dynamics below).  In a related work, Coville et al. \cite{CDM2008}  showed that asymmetric kernels may induce  nonpositive minimal wave speed which  always coincides with spreading speed in the Fisher-KPP case. However, they did not point out when the minimal wave speed is nonpositive.

We find that the spreading speed to the left $c_l^*$  has the same sign as that of $E(k)-f'(0)$ and the spreading speed to the right $c_r^*$ has the same sign as that of $E(k)+f'(0)$. Here $E(k)$ stands for the asymmetry level of $k$ and is defined by
\[
E(k)\triangleq\text{sign}(J(k))\left[1-\inf\limits_{\lambda\in\mathbb R}\left\{\int_{\mathbb R}k(x)e^{\lambda x}dx\right\}\right],
\]
where $J(k)\triangleq\int_{\mathbb R}k(x)xdx$ is the first moment and  $k$ belongs to the set  that consists of all nonnegative and continuous functions satisfying (K1) and $\int_{\mathbb R}k(x)dx=1$.
From this result, we show that  asymmetric dispersal influences the signs of spreading speeds, and further  influences the spatial dynamics
in three aspects: it can determine the spatial propagation directions of solutions, influence the stability of equilibrium states, and affect the monotone property of solutions. More details are given in Section 2.

The results are applied to two special cases where  $k$ is a normal distribution and a uniform distribution, respectively. We present more details of the calculation of  $E(k)$ and show how the asymmetric dispersal influences spatial dynamics in Section 5.

Second, we give an improved proof of the spreading speed result. In \cite{LPL2005}, Lutscher et al. proved the  spreading speed result by constructing an innovative lower solution of nonlocal dispersal equation \eqref{1.1}, which can spread at any speed $c$ in $(c_l^*,c_r^*)$, as follows
\begin{equation}\label{1.11}
\underline u (t,x)=
\left\{
\begin{aligned}
&\varepsilon e^{-s(x-ct)}\sin(\gamma(x-ct)),&&x-ct\in[0,\pi/\gamma],\\
&0,&&x-ct>\pi/\gamma.
\end{aligned}
\right.
\end{equation}
In the construction of this lower solution,  they needed to make some technical requirements on $k$. For example, they assumed that $\text{supp}(k)=\mathbb R$ and the function $x\mapsto\exp(sx)k(x)$ is decreasing for large enough $x$. They also made some requirements on the monotone property  of the function $A(s)=(\int_{\mathbb R}k(x)e^{sx}dx-1+f'(0))/s$, $s\neq0$.

In this paper, without any additional assumptions, we construct two new lower solutions  which   spread  at   speeds of $c_1$  and $c_2$, respectively, as follows
\begin{equation}\label{1.3}
\underline u_i(t,x)=\max\{0,H_i(e^{\rho_i(-x+c_it+\xi_i)})\},~~i=1,2,\\
\end{equation}
with
\[
H_i(z)=A_iz-B_iz^{1+\delta_i}-D_iz^{1-\delta_i},~~z>0,
\]
where   $c_1\in(c_r^*-\epsilon,c_r^*)$ and $c_2\in(c_l^*,c_l^*+\epsilon)$ for small $\epsilon>0$.

However,  some property of \eqref{1.3}  is not as good as that of \eqref{1.11}, because the speed of \eqref{1.3} is limited to $(c_l^*,c_l^*+\epsilon)$ or $(c_r^*-\epsilon,c_r^*)$.
Therefore, we   give a new method to study the whole situation of $(c_l^*,c_r^*)$, which is called the ``\textbf{forward-backward spreading}'' method.
In this method, for any $\tau>0$ we divide the  time  period of  $[0,\tau]$ into two parts $[0,\kappa\tau]$ and $[\kappa\tau,\tau]$, where $\kappa$ is any   number in $[0,1]$.
In $[0,\kappa\tau]$ we construct a lower solution $u_1(t,x)$ spreading at a speed of $c_1\in(c_r^*-\epsilon,c_r^*)$.  In $[\kappa\tau,\tau]$ we construct another lower solution $u_2(t,x)$ which  spreads at a speed of $c_2\in (c_l^*,c_l^*+\epsilon)$ and satisfies  that $u_2(\kappa\tau,x)\leqslant u_1(\kappa\tau,x)$.
Then these two lower solutions can be regarded as a lower solution  defined in  $[0,\tau]$ whose speed is $\bar{c}=\kappa c_1+(1-\kappa)c_2$. Moreover, the arbitrariness of $\kappa$ ensures that $\bar{c}$ can be equal to any number in $[c_1, c_2]$.
We remark that the term ``forward-backward spreading'' comes from the special case    $c_l^*<0<c_r^*$, which means $u_1(t,x)$ spreads forward and $u_2(t,x)$ spreads backward.

By constructing the new lower solutions and applying the ``forward-backward spreading'' method, we improve the proof of spreading speed  result and further obtain a property about the spatial propagation of solutions  (see Corollary \ref{th2.5}).
\begin{remark}\rm
In the study of traveling wave solutions, we usually  construct the lower solution $\underline v(t,x)=\max\{0,e^{\rho(-x+ct)}-Le^{\rho(1+\delta)(-x+ct)}\}$ where $L$ is large enough. Note that $\underline v(t,x)>0$ for  $x$ large enough.
Different from $\underline v(t,x)$,   the lower solutions defined by \eqref{1.3} have no tails on two sides, which means that the function $\underline u_i(t,\cdot)$ is compactly supported. Therefore, the lower solutions defined by \eqref{1.3} can be used to study the spreading speed for compactly supported  initial data.
\end{remark}

Third, we establish the relationship between spreading speed and exponentially decaying initial data. In a reaction-diffusion equation, it is well-known that the decay behavior to zero as $x\rightarrow\pm\infty$ of the initial data  influences  the spreading speed, see e.g.  Booty et al. \cite{BHM1993}, Hamel and Nadin \cite{HN2012}, McKean \cite{McK1975}, Sattinger \cite{Sat1976}. Moreover, when the initial datum is exponentially unbound or the kernel is ``fat-tailed'', the  acceleration propagation (namely, its spreading speed approaches to infinity as $t\rightarrow+\infty$)   is studied by  Alfaro \cite{Alfaro2017}, Alfaro and Coville \cite{AC2016}, Finkelshtein et al. \cite{FKT2016}, Finkelshtein and Tkachov \cite{FT2017}, Garnier \cite{Garnier},   Hamel and Roques \cite{HR2010}, Xu et al. \cite{XLL2018, XLR2018}. Therefore, it is  necessary to study the  influence of initial data on the spreading speed of equation \eqref{1.1}.

Here we consider the exponentially decaying initial function which satisfies that
\[
u_0(x)\sim O(e^{-\lambda |x|})~~\text{as}~|x|\rightarrow +\infty.
\]
When $k$ is symmetric, for $\lambda\in[\lambda^*,+\infty)$ the spreading speed  of equation \eqref{1.1} is $c^*\triangleq c_r^*=-c_l^*$, and for $\lambda\in(0,\lambda^*)$ the spreading speed equals to
\[
c(\lambda)=\lambda^{-1}\Big[\int_{\mathbb R}k(x)e^{\lambda x}dx-1+f'(0)\Big] ~~\text{for}~\lambda\neq0.
\]
Moreover,  $c(\lambda)$ decreases strictly along with the increase of  $\lambda\in(0,\lambda^*)$ and we have $c^*=c(\lambda^*)$.

The rest of this paper is organized as follows. In Section 2, we study the signs of spreading speeds and the influences of asymmetric dispersal on  spatial dynamics. Section 3 presents the new lower solutions and  the new ``forward-backward spreading'' method.  By using them, we give an improved proof of the  spreading speed result. Section 4 deals with the relationship between spreading speed and   exponentially decaying initial data. In Section 5, two examples are provided to explain the   results on the signs of spreading speeds.

\section{The signs of spreading speeds}

In this section we present the main results about the signs of spreading speeds and the influences of asymmetric dispersal on the spatial dynamics.

First we introduce some notations.  Let $c(\cdot)$ be the function defined by
\begin{equation}\label{2.1}
c(\lambda)=\lambda^{-1}\Big[\int_{\mathbb R}k(x)e^{\lambda x}dx-1+f'(0)\Big] ~~\text{for}~\lambda\in(\lambda^-,0)\cup(0,\lambda^+),
\end{equation}
where
\begin{eqnarray}
&& \lambda^+=\sup\Big\{\lambda>0~\big|\int_\mathbb R k(x)e^{\lambda x}dx<+\infty\Big\}\in\mathbb R^+\cup\{+\infty\}, \label{1.9} \\
&& \lambda^-=\inf\Big\{\lambda<0~\big|\int_\mathbb R k(x)e^{\lambda x}dx<+\infty\Big\}\in\mathbb R^-\cup\{-\infty\}. \label{1.10}
\end{eqnarray}
When $k$ satisfies (K1) and (K2),  by $\frac{\partial^2}{\partial \lambda^2}\int_{\mathbb R}k(x)e^{\lambda x}dx>0$, we can find a unique constant $\lambda(k)\in(\lambda^-,\lambda^+)$ such that
\[
\int_{\mathbb R}k(x)e^{\lambda(k) x}dx=\inf\limits_{\lambda\in\mathbb R}\int_{\mathbb R}k(x)e^{\lambda x}dx,~\text{namely}~~\int_{\mathbb R}k(x)e^{\lambda(k) x}xdx=0.
\]
Since the  function  $\lambda \mapsto \int_{\mathbb R}k(x)e^{\lambda  x}xdx$ is  strictly increasing, it holds that
\begin{equation}\label{2.8}
\int_{\mathbb R}k(x)e^{\lambda x}xdx>0~\text{for}~\lambda>\lambda(k)~\text{and}~\int_{\mathbb R}k(x)e^{\lambda x}xdx<0~\text{for}~\lambda<\lambda(k).
\end{equation}
It follows from $J(k)=\left.\int_{\mathbb R}k(x)e^{\lambda  x}xdx\right|_{\lambda=0}$ that
$\text{sign}(J(k))=-\text{sign}(\lambda(k))$.
Then we have that
\begin{equation}\label{2.9}
E(k)=-\text{sign}(\lambda(k))\left[1-\int_{\mathbb R}k(x)e^{\lambda(k) x}dx\right].
\end{equation}

Now we state two properties of $E(k)$ and show that the spreading speed to the left $c_l^*$  has the same sign as that of $E(k)-f'(0)$ and the spreading speed to the right  $c_r^*$  has the same sign as that of  $E(k)+f'(0)$.

\begin{proposition}\label{pro1.1}
The function $E(k)$ satisfies that
\begin{itemize}
\item[(i)] $E(k)=-E(\check{k})$, where $\check{k}(x)=k(-x)$ for $x\in\mathbb R$;
\item[(ii)] If $k_1$ is more skewed to the right than $k_2$, then $E(k_1)\geqslant E(k_2)$. Here the  concept that $k_1$ is more skewed to the right than $k_2$ means that $k_1(x)\geqslant k_2(x)$ for $x\in\mathbb R^+$ and $k_1(x)\leqslant k_2(x)$ for $x\in\mathbb R^-$.
\end{itemize}
\end{proposition}

\begin{proof}
Since $J(k)=-J(\check{k})$ and
\[
\inf\limits_{\lambda\in\mathbb R}\int_{\mathbb R}k(x)e^{\lambda x}dx=\inf\limits_{\lambda\in\mathbb R}\int_{\mathbb R}\check{k}(x)e^{\lambda x}dx,
\]
we have that  $E(k)=-E(\check{k})$.

Now suppose that $k_1(x)\geqslant k_2(x)$ for $x\in\mathbb R^+$ and $k_1(x)\leqslant k_2(x)$ for $x\in\mathbb R^-$. Denote  $\lambda_1\triangleq\lambda(k_1)$ and $\lambda_2\triangleq\lambda(k_2)$.
By $\int_{\mathbb R}(k_1(x)-k_2(x))e^{\lambda_2x}xdx\geqslant0$, we get from $\int_{\mathbb R}k_i(x)e^{\lambda_ix}xdx=0$   that
\[
\int_{\mathbb R}k_1(x)e^{\lambda_2x}xdx\geqslant0=\int_{\mathbb R}k_1(x)e^{\lambda_1x}xdx.
\]
Notice that the function  $\lambda \mapsto \int_{\mathbb R}k_1(x)e^{\lambda  x}xdx$ is increasing, then $\lambda_1\leqslant \lambda_2$. Now we consider three cases.
First,  when $\lambda_1\leqslant 0\leqslant\lambda_2$,  we easily check that $E(k_1)\geqslant0\geqslant E(k_2)$ by \eqref{2.9}. Next, consider  the case  $\lambda_1\leqslant\lambda_2 \leqslant0$. Some calculations imply that
\[
\begin{aligned}
&E(k_1)=1-\int_{\mathbb R}k_1(x)e^{\lambda_1x}dx=\int_{\lambda_1}^0\left[\int_{\mathbb R}k_1(x)e^{\lambda x}xdx\right]d\lambda,\\
&E(k_2)=1-\int_{\mathbb R}k_2(x)e^{\lambda_2x}dx=\int_{\lambda_2}^0\left[\int_{\mathbb R}k_2(x)e^{\lambda x}xdx\right]d\lambda.
\end{aligned}
\]
We have that
\[
E(k_1)-E(k_2)=\int_{\lambda_1}^{\lambda_2}\left[\int_{\mathbb R}k_1(x)e^{\lambda x}xdx\right]d\lambda+\int_{\lambda_2}^0\left[\int_{\mathbb R}(k_1(x)-k_2(x))e^{\lambda x}xdx\right]d\lambda.
\]
It follows from  \eqref{2.8} that $\int_{\mathbb R}k_1(x)e^{\lambda x}xdx>0$ for $\lambda>\lambda_1$.
Then we  obtain   $E(k_1)\geqslant E(k_2)$ by $\int_{\mathbb R}(k_1(x)-k_2(x))e^{\lambda x}xdx\geqslant0$. Finally,  in the case  $0\leqslant\lambda_1\leqslant\lambda_2$, we can prove  $E(k_1)\geqslant E(k_2)$  by a similar method.
\end{proof}

From Proposition \ref{pro1.1}, we can use $E(k)$ to describe  the asymmetry level of $k$.
Indeed, it is easy to check that  $E(k)\in[-1,1]$ by $\int_{\mathbb R} k(x)dx=1$ and $k(x)\geqslant 0$, $x\in\mathbb R$.
In particular, if $k(\cdot)$ is  symmetric, then  $E(k)=0$; and
if $k(x)=0$ for all $x\in\mathbb R^+$, then $E(k)=1$.
Similarly, if $k(x)=0$ for all $x\in\mathbb R^-$, then $E(k)=-1$.
Moreover, when $E(k)>0$, $k$ can be regarded as a function skewed to the right  and when $E(k)<0$, it is a function skewed to the left.

\begin{remark}\rm
The properties (i) and (ii) in Proposition \ref{pro1.1} are two fundamental requirements for the function describing  the asymmetry level of $k$. For example, consider  $E_g(k)\triangleq \int_{\mathbb R} k(x)g(x)dx$, where $g$ is an odd    function and is positive in $\mathbb R^+$. Then we can use $E_g$  to describe the asymmetry level of $k$ too.
It is  easy to check that $E_g(k)$ satisfies (i) and (ii).
A special form of $E_g(k)$ is given by the moment function $\int_{\mathbb R}k(x)x^Ndx$, where $N$ is an odd number.
\end{remark}

The following lemma will be used several times in the remainder  of the paper.

\begin{lemma}\label{th2.99}
For any $k(\cdot)$ satisfying (K1) and (K2), there are   unique  $\lambda_r^*\in(0,\lambda^+)$ and $\lambda_l^*\in(\lambda^-,0)$ such that
\begin{equation}\label{2.2}
c_r^*=\inf\limits_{\lambda\in(0,\lambda^+)}\{c(\lambda)\}=c(\lambda_r^*)=\int_{\mathbb R}k(x)e^{\lambda_r^* x}xdx.
\end{equation}
and
\begin{equation}\label{2.3}
c_l^*=\sup\limits_{\lambda\in(\lambda^-,0)}\{c(\lambda)\}=c(\lambda_l^*)=\int_{\mathbb R}k(x)e^{\lambda_l^* x}xdx.
\end{equation}
\end{lemma}
\begin{proof}
For $\lambda\in(\lambda^-,0)\cup(0,\lambda^+)$,  a simple calculation implies that
\[
c'(\lambda)=\lambda^{-1}\int_\mathbb R k(x)e^{\lambda x}xdx-\lambda^{-2}\Big[\int_{\mathbb R}k(x)e^{\lambda x}dx-1+f'(0)\Big].
\]
We get  that $\lim\limits_{\lambda^\rightarrow 0^+}c'(\lambda)=-\infty$. Next, we   show that
\begin{equation}\label{2.5}
c'(\lambda)>0~~\text{for any $\lambda$ close  to $\lambda^+$}.
\end{equation}
In the case   $\lambda^+<+\infty$, let $M$ be a positive constant satisfying that $M\lambda^+>1$. Then  there are two constants $C_1$ and $C_2$ such that for any $\lambda\in(0,\lambda^+)$,
\[
\int_\mathbb R k(x)e^{\lambda x}xdx\geqslant M\int_M^{+\infty}k(x)e^{\lambda x}dx+C_1,~~ \int_\mathbb R k(x)e^{\lambda x}dx\leqslant \int_M^{+\infty}k(x)e^{\lambda x}dx+C_2.
\]
Then we can obtain  \eqref{2.5}  by $\int_M^{+\infty}k(x)e^{\lambda x}dx\rightarrow+\infty$ as $\lambda\rightarrow\lambda^+$. In the case   $\lambda^+=+\infty$, we need to rewrite $c'(\lambda)$ by
\[
c'(\lambda)=\lambda^{-2}\Big[\int_\mathbb R k(x)e^{\lambda x}(\lambda x-1)dx+1-f'(0)\Big].
\]
Then  we get \eqref{2.5} by $e^{\lambda x}\geqslant \lambda x+1$, $x\in\mathbb R$.
On the other hand, when $c'(\lambda)=0$, it follows that  $c''(\lambda)>0$ for $\lambda\in(0,\lambda^+)$. Therefore, there is a unique constant $\lambda_r^*\in(0,\lambda^+)$ such that
\[
c'(\lambda_r^*)=0~\text{and}~ c(\lambda_r^*)=\inf\limits_{\lambda\in(0,\lambda^+)}\{c(\lambda)\}=\int_{\mathbb R}k(x)e^{\lambda_r^* x}xdx.
\]
Moreover, we have that
\[
c'(\lambda)<0~\text{for}~\lambda\in(0,\lambda_r^*)~~\text{and}~c'(\lambda)>0~\text{for}~\lambda\in(\lambda_r^*,\lambda^+).
\]
Similarly, the existence and uniqueness of $\lambda_l^*$ can  be obtained.
\end{proof}

\begin{theorem}\label{th1.2}
Suppose  that {\rm(H)}, {\rm(K1)} and {\rm{(K2)}} hold.
Then $c_l^*$ and $c_r^*$ satisfy that
\begin{itemize}
  \item[(i)] if $E(k)>f'(0)$, then $0<c_l^*<c_r^*${\rm;}
  \item[(ii)] if $E(k)=f'(0)$, then $0=c_l^*<c_r^*${\rm;}
  \item[(iii)] if $-f'(0)<E(k)<f'(0)$, then $c_l^*<0<c_r^*${\rm;}
  \item[(iv)] if $E(k)=-f'(0)$, then $c_l^*<c_r^*=0${\rm;}
  \item[(v)] if $E(k)<-f'(0)$, then $c_l^*<c_r^*<0$.
\end{itemize}
\end{theorem}

\begin{proof}
By \eqref{2.2} and \eqref{2.3}, it is  easy to check that $c_l^*<c_r^*$, since  the  function  $\lambda \mapsto \int_{\mathbb R}k(x)e^{\lambda  x}xdx$ is  strictly increasing.
When $E(k)>f'(0)$, by \eqref{2.9} we get that $\lambda(k)<0$ and
\[
E(k)=1-\int_{\mathbb R}k(x)e^{\lambda(k)x}dx>f'(0).
\]
From  \eqref{1.7} it follows that
\[
c_l^*\geqslant \lambda(k)^{-1}\Big[\int_{\mathbb R}k(x)e^{\lambda(k) x}dx-1+f'(0)\Big]>0.
\]
Then case (i) is proved.
When $E(k)=f'(0)$, it is easy to check  that $\lambda(k)<0$ and
\[
E(k)=1-\inf\limits_{\lambda\in\mathbb R}\left\{\int_{\mathbb R}k(x)e^{\lambda x}dx\right\}=1-\int_{\mathbb R}k(x)e^{\lambda(k)x}dx=f'(0).
\]
It implies that $\int_{\mathbb R}k(x)e^{\lambda x}dx-1+f'(0)\geqslant0$ for any $\lambda\in\mathbb R$. Then  we obtain $c_l^*=0$ by  \eqref{1.7}. Case (ii) is proved.
When $-f'(0)<E(k)<f'(0)$, we have that
\[
1-\inf\limits_{\lambda\in\mathbb R}\left\{\int_{\mathbb R}k(x)e^{\lambda x}dx\right\}<f'(0).
\]
From \eqref{1.7} and \eqref{1.8} it follows that $c_l^*<0<c_r^*$. Then case (iii) is proved.
Finally, the proofs of cases (iv) and (v) are similar to those of cases (ii) and (i), respectively.
\end{proof}

Combining  Theorem \ref{th1.2} with \eqref{1.2}, we see that   $E(k)$ and  $f'(0)$ can determine the signs of spreading speeds. Moreover, they have three   important  influences on the spatial dynamics of nonlocal dispersal equation \eqref{1.1}.

First, the signs of $c_r^*$ and $c_l^*$ can determine the spatial propagation directions  of solution.
Define a level set function  by
\[
\Sigma_\omega(t)\triangleq\{x\in\mathbb R~|~u(t,x)\geqslant \omega\}~~\text{for any}~\omega\in(0,1),~t>0.
\]
Then when $t$ is large enough, $\Sigma_\omega(t)$ spreads to both the left and the right of the $x$-axis in case (iii), spreads only to the right in case (i), and spreads only to the left in case (v).
However, in  case (ii), if the set $\Sigma_\omega(t)$ is connected, the movement of the left boundary of $\Sigma_\omega(t)$ is slower than linearity and we cannot identify its propagating direction. Similarly, we cannot  identify the propagating direction of  the right boundary of $\Sigma_\omega(t)$ in case (iv) either.

Next, the signs of $c_r^*$ and $c_l^*$  influence the stability of equilibrium states.
In case (iii), the equilibrium state  $u\equiv1$  is globally stable and $u\equiv0$ is  globally unstable in any bounded spatial
region. More precisely, if $u_0(\cdot)\not\equiv0$  and $u_0$ is continuous and nonnegative, then
\[
\text{for any}~x\in\mathbb R,~\lim\limits_{t\rightarrow\infty} u(t,x)=1;
\]
namely, case (iii) has the same stability property as \eqref{1.5} in reaction-diffusion equations.
However, in  case  (i) or (v),  the equilibrium state $u\equiv0$ becomes  stable in any bounded spatial region for  compactly supported initial data, which means that
\[
\text{for any}~x\in\mathbb R,~\lim\limits_{t\rightarrow\infty} u(t,x)=0.
\]
The fundamental reason  of this change is that the asymmetric dispersal plays a more important role than the reaction term, and the spatial region $\Sigma_\omega(t)$   travels in the dominating direction of dispersal (namely, the direction of $\text{sign}(J(k))$).
In addition, it is worth pointing out that the equilibrium state $u\equiv0$ remains to be unstable for  initial data satisfying $u_0(x)\geqslant \epsilon$ with $\epsilon>0$ (see Finkelshtein et al. \cite{FKT2015,FKMT2017}).

Finally, the asymmetry of $k$  affects the monotone property of  solutions.
In the reaction-diffusion equation \eqref{1.4}, there is a well-known result that the solution keeps some symmetry  and monotone property of the initial data; that is, if $u_0(\cdot)$ is symmetric   and decreasing on $\mathbb R^+$, so is  the solution $u(t,\cdot)$ of equation \eqref{1.4} at any time $t>0$. The following theorem shows that this result also holds in the nonlocal dispersal equation.

\begin{theorem}\label{th2.4}
If $k(\cdot)$ and $u_0(\cdot+x_1)$ are symmetric  and decreasing on $\mathbb R^+$ with $x_1\in\mathbb R$, so is  the solution $u(t,\cdot+x_1)$ of equation \eqref{1.1}  at any time $t>0$.
\end{theorem}

\begin{proof} By translating the $x$-axis, we suppose that $x_1=0$.
The symmetry property of $u(t,\cdot)$  can be obtained easily. Indeed,  if we consider the following equation
\[
\left\{
\begin{aligned}
&v_t(t,x)= k*v(t,x)-v(t,x)+ f(v(t,x)),~~t>0,~x\in \mathbb R,\\
&v(0,x)=u_0(-x),~x\in\mathbb R,
\end{aligned}
\right.
\]
then $u(t,x)=v(t,x)=u(t,-x)$ for $t\geqslant0$ and $x\in\mathbb R$ by the uniqueness of the solution.
For a fixed number $y\in \mathbb R^+$, we define
\[
w(t,x)=u(t,x+2y)-u(t,x)~~\text{for}~t\geqslant0,~x\in\mathbb R.
\]
The symmetry property of $u(t,\cdot)$   implies that  $w(t, -y)=0$ for $t\geqslant0$ and
\begin{equation}\label{5.2}
w(t,x)=-w(t,-x-2y)~~\text{for}~t\geqslant0,~x\in\mathbb R.
\end{equation}
Since  $u_0(\cdot)$ is  symmetric   and  decreasing on $\mathbb R^+$, we have
\[
w(0,x)\leqslant 0~\text{for}~x>-y,~~w(0,x)\geqslant 0~\text{for}~x<-y.
\]

In order to prove that $u(t,\cdot)$ is  decreasing on $\mathbb R^+$ for any $t>0$, we try to prove the following conclusion
\begin{equation}\label{5.5}
w(t,x)\leqslant 0~\text{for}~t>0,~x>-y.
\end{equation}
Indeed, if \eqref{5.5} holds, then we have that $u(t,x+2y)\leqslant u(t,x)$ for $x>-y$ and $y\in\mathbb R^+$ at any time $t>0$, which means that $u(t,\cdot)$ is  decreasing on $\mathbb R^+$.

Now we begin to prove \eqref{5.5}.
Since $f(u)\in C^1([0,1])$, there is a constant $M>0$ such that
\begin{equation}\label{5.6}
\begin{aligned}
w_t(t,x)&=k*w(t,x)-w(t,x)+f(u(t,x+2y))-f(u(t,x))\\
&\leqslant k*w(t,x)-w(t,x)+Mw(t,x) ~~~~\text{for}~t>0,~x\in\mathbb R.
\end{aligned}
\end{equation}
Suppose by contradiction that \eqref{5.5} does not hold and there exist two constants $T_0>0$ and $\varepsilon>0$ such that
\begin{equation}\label{5.7}
\sup\limits_{x>-y}\{w(T_0,x)\}=\varepsilon e^{KT_0}~~\text{and}~w(t,x)<\varepsilon e^{Kt}~\text{for}~t\in(0,T_0),~x>-y,
\end{equation}
where $K>\max\{M+1,~\frac{8}{3}M+\frac{4}{3}\}$.
Under \eqref{5.7} we give an estimate for the nonlocal dispersal term  $k*w(t,x)-w(t,x)$.
From  \eqref{5.2}, it follows that for $t\in(0,T_0]$ and $x>-y$,
\[
\begin{aligned}
&k*w(t,x)-w(t,x)\\
&\quad\quad =\int_{-y}^{+\infty}[w(t,z)-w(t,x)]k(x-z)dz+\int_{-\infty}^{-y}[w(t,z)-w(t,x)]k(x-z)dz\\
&\quad\quad =\int_{-y}^{+\infty}Q(t,x,z,y)dz\\
&\quad\quad =\int_{\Sigma_1(t)}Q(t,x,z,y)dz+\int_{\Sigma_2(t)}Q(t,x,z,y)dz,
\end{aligned}
\]
where
\[
Q(t,x,z,y)
=[w(t,z)-w(t,x)]k(x-z)-[w(t,z)+w(t,x)]k(x+z+2y)
\]
and
\[
\Sigma_1(t)=\big\{z~|~w(t,z)>0,~z>-y\big\},~~\Sigma_2(t)=\big\{z~|~w(t,z)\leqslant0,~z>-y\big\}.
\]
We also suppose that $w(t,x)\geqslant 0$ in this estimation. When $z\in\Sigma_1(t)$, we can get from \eqref{5.7} that $w(t,z)-w(t,x)\leqslant \varepsilon e^{Kt}$. Then it follows that
\[
\int_{\Sigma_1(t)}Q(t,x,z,y)dz\leqslant\int_{\Sigma_1(t)}\varepsilon e^{Kt}k(x-z)dz\leqslant \varepsilon e^{Kt}~~\text{for}~t\in(0,T_0],~x>-y.
\]
When $z\in\Sigma_2(t)$, we rewrite $Q(t,x,z,y)$ as
\[
Q(t,x,z,y)
=w(t,z)[k(x-z)-k(x+z+2y)]-w(t,x)[k(x-z)+k(x+z+2y)].
\]
Since $k(\cdot)$ is  symmetric and decreasing on $\mathbb R^+$, we have that $k(x-z)-k(x+z+2y)\geqslant 0$ when $x>-y$ and $z>-y$. Then it follows that
\[
\int_{\Sigma_2(t)}Q(t,x,z,y)dz\leqslant 0~~\text{for}~t\in(0,T_0],~x>-y.
\]
Therefore, when $w(t,x)\geqslant 0$, we have  that
\begin{equation}\label{5.8}
k*w(t,x)-w(t,x)\leqslant \varepsilon e^{Kt}~~\text{for}~t\in(0,T_0],~x>-y.
\end{equation}

Next we return to the proof of  \eqref{5.5}.  From \eqref{5.7}, the continuity property of $w(T_0,\cdot)$ implies that one or both of the following two cases must happen.
\begin{description}
  \item[Case 1] There exists $x_0\in(-y,+\infty)$ such that $w(T_0,x_0)=\sup\limits_{x>-y}\left\{w(T_0,x)\right\}=\varepsilon e^{KT_0}$.
  \item[Case 2] It holds that $\limsup\limits_{x\rightarrow+\infty}\{w( T_0,x)\}=\varepsilon e^{KT_0}$.
\end{description}

If case 1 holds,   from  \eqref{5.7} we have
\[
\left.\frac{\partial}{\partial t}\left(w(t,x_0)-\varepsilon e^{Kt}\right)\right|_{t=T_0}\geqslant 0,
\]
which implies
\begin{equation}\label{5.9}
w_t(T_0,x_0)\geqslant\varepsilon K e^{KT_0}.
\end{equation}
From \eqref{5.8} and \eqref{5.9}, it follows that
\[
w_t(T_0,x_0)-k*w(T_0,x_0)+w(T_0,x_0)-Mw(T_0,x_0)\geqslant (K-1-M)\varepsilon  e^{KT_0}>0,
\]
which contradicts \eqref{5.6}.

If case 2 holds, then there exists a constant number $x_1>-y$ (far away from $-y$) such that  $w(T_0,x_1)>\frac{3}{4}\varepsilon e^{KT_0}$. Let $p_0(x)$ be a smooth and  increasing function satisfying that
\[
p_0(x)=\left\{
\begin{aligned}
&1~~~\text{for}~x\leqslant x_1,\\
&3~~~\text{for}~x\geqslant x_1+1.
\end{aligned}
\right.
\]
For $\sigma>0$, we define
\[
\rho_{\sigma}(t,x)= \left[\frac{1}{2}+\sigma p_0(x)\right]\varepsilon e^{Kt}~\text{for}~t\in[0,T_0],~x\in\mathbb R
\]
and
\[
\sigma^*=\inf\Big\{\sigma>0~|~w(t,x)-\rho_{\sigma}(t,x)\leqslant0 ~~\text{for}~t\in[0,T_0],~x>-y \Big\}.
\]
From \eqref{5.7}, some simple calculations yield  that $\rho_\frac{1}{2}(t,x)\geqslant\varepsilon e^{Kt}\geqslant w(t,x)$ for $t\in[0,T_0]$ and $x>-y$ and
$\rho_\frac{1}{4}(T_0,x_1)=\frac{3}{4}\varepsilon e^{KT_0}<w(T_0,x_1)$.
Then we have that $\frac{1}{4}\leqslant\sigma^*\leqslant\frac{1}{2}$ and
\[
\rho_{\sigma^*}(t,x)\geqslant\frac{5}{4}\varepsilon e^{Kt}>w(t,x)~~\text{for}~t\in[0,T_0],~x\geqslant x_1+1.
\]
From the definition of $\sigma^*$,  there must exist $T_1\in(0,T_0]$ and $x_2\in(-y,x_1+1)$ such that
\[
w(T_1,x_2)-\rho_{\sigma^*}(T_1,x_2)=\sup\limits_{t\in[0,T_0],~x>-y}\big\{w(t,x)-\rho_{\sigma^*}(t,x)\big\}=0,
\]
which implies that
\[
\frac{\partial}{\partial t}\Big(w(t,x_2)-\rho_{\sigma^*}(t,x_2)\Big)\bigg|_{t=T_1}\geqslant 0.
\]
Since $\frac{1}{4}\leqslant\sigma^*\leqslant\frac{1}{2}$, we have
\begin{equation}\label{5.10}
w(T_1,x_2)=\rho_{\sigma^*}(T_1,x_2)\leqslant\rho_{\frac{1}{2}}(T_1,x_2)\leqslant 2\varepsilon e^{KT_1}
\end{equation}
and
\begin{equation}\label{5.11}
w_t(T_1,x_2)\geqslant \frac{\partial}{\partial t}\rho_{\sigma^*}(T_1,x_2)=K\rho_{\sigma^*}(T_1,x_2)\geqslant K\rho_{\frac{1}{4}}(T_1,x_2)\geqslant \frac{3}{4}K\varepsilon e^{KT_1}.
\end{equation}
From \eqref{5.8}, \eqref{5.10} and \eqref{5.11}, we can get that
\[
w_t(T_1,x_2)-k*w(T_1,x_2)+w(T_1,x_2)-Mw(T_1,x_2)\geqslant (\frac{3}{4}K-1-2M)\varepsilon  e^{KT_1}>0,
\]
which contradicts \eqref{5.6}.

Finally, we get \eqref{5.5}  and the proof of Theorem \ref{th2.4} is finished.
\end{proof}

However, when $k$ is  asymmetric,   Theorem \ref{th2.4} does not hold  even if $k$ has an adequate monotone property. For example, in  case (i) or (v),  the spatial point where the solution attains  its maximum value keeps moving at a speed between $c_l^*$ and $c_r^*$. We also point out that   Theorem \ref{th2.4} is useful in Remark \ref{re3.5} and  the proof of Theorem \ref{th2.3}.

Recently, we \cite{XLR2019} further study the relationship between the signs of spreading speeds and  the asymmetric dispersals of infectious agents and infectious humans in an epidemic model, where the infectious agents are brought by migratory birds. We find it possible that the epidemic spreads only
along the flight route of migratory birds and the spatial propagation against the flight route
fails, as long as the infectious humans are kept from moving frequently.

\begin{remark}\rm
In reaction-diffusion equation \eqref{1.4}, the proof of  the same  conclusion as Theorem \ref{th2.4} is easier than that in  nonlocal dispersal equation. Indeed, we can prove \eqref{5.5} by the maximum principle of equation $w_t(t,x)= \Delta w(t,x)+Mw(t,x)$ with $(t,x)\in[0,+\infty)\times[-y,+\infty)$.
\end{remark}

\section{Improved  proof of spreading speeds}
\noindent

In this section,   we give an improved  proof of the  spreading speed  result in equation \eqref{1.1} by constructing  new lower solutions and applying  the ``forward-backward spreading'' method.
First, we  state the comparison principle  (see e.g. \cite{Chen2002,CDM2008}).

\begin{lemma}[Comparison principle]\label{th3.1}
Suppose the bounded continuous functions $\bar u(t,x)$ and $\underline u(t,x)$ are the upper and lower solutions of equation \eqref{1.1} for $t\in(0,T]$, in the sense that
\[
\bar u_t-k*\bar u+u- f(\bar u)\geqslant0 \geqslant \underline u_t-k*\underline u+\underline u- f(\underline u)~~\text{for}~t\in(0,T],~x\in\mathbb R.\\
\]
If $\bar u(0,x)\geqslant \underline u(0,x)$ for $x\in \mathbb R$, then $\bar u(t,x)\geqslant \underline u(t,x)$ for $t\in[0,T]$ and $x\in \mathbb R$.
\end{lemma}

In the construction of the new lower solutions, we need an auxiliary function and  some of its properties as stated in the following lemma.

\begin{lemma}\label{th3.4}
For any $\delta\in(0,1)$,    define
\[
H(z)=Az-Bz^{1+\delta}-Dz^{1-\delta}~~\text{for}~z>0.
\]
For any given $A>0$ and $D>0$,  we have the following conclusions
\[
\begin{aligned}
&H^{\rm{max}}>0~~\text{for}~B\in\left(0,A^2/(4D)\right),\\
&H^{\rm{max}}\rightarrow 0^+,~\nu-\mu\rightarrow0^+~~\text{as}~B-A^2/(4D)\rightarrow 0^-,
\end{aligned}
\]
where
\[
\begin{aligned}
&H^{\rm{max}}\triangleq\sup\limits_{z>0}\big\{H(z)\big\}=H(z_0)~~\text{for some}~z_0\in(\mu,\nu),\\
&(\mu,\nu)\triangleq\big\{z>0~|~H(z)>0\big\}~~\text{for}~B\in\left(0, A^2/(4D)\right).
\end{aligned}
\]
Moreover, for any  $p>0$, there exists $B(p)\in\left(0, A^2/(4D)\right)$ such that
\[
H^{\rm{max}}=p~\text{and}~B(p)\rightarrow  A^2/(4D)~\text{as}~p\rightarrow0^+.
\]
\end{lemma}

\begin{proof}
For any given $A>0$, $D>0$ and $\delta\in(0,1)$,   define
\[
h(z,B)=Az-Bz^{1+\delta}-Dz^{1-\delta}~~~\text{for}~z>0,~0<B\leqslant\frac{A^2}{4D(1-\delta^2)}.
\]
Let $z_0$ and  $z_1$ be two positive numbers given by
\[
z_0^{\delta}=\frac{A+\sqrt{A^2-4BD(1-\delta^2)}}{2B(1+\delta)}~~\text{and}~~
z_1^{\delta}=\frac{A-\sqrt{A^2-4BD(1-\delta^2)}}{2B(1+\delta)}.
\]
A simple  calculation  implies that
\[
\frac{\partial}{ \partial z}h(z,B)\left\{
\begin{aligned}
&=0~~\text{for}~z=z_1~\text{and}~z=z_0,\\
&<0~~\text{for}~z\in (0,z_1)\cup(z_0,+\infty),\\
&>0~~\text{for}~z\in(z_1,z_0).
\end{aligned}
\right.
\]
Therefore, we have  $H^{\rm{max}}=\max\{0,h(z_0,B)\}$. Define
\[
g(B)=h(z_0,B)=Az_0-Bz_0^{1+\delta}-Dz_0^{1-\delta}~~~\text{for}~0<B\leqslant\frac{A^2}{4D(1-\delta^2)}.
\]
Then it follows that $g'(B)=-z_0^{1+\delta}<0$. Notice that
\begin{equation}\label{5.1}
g(B)=0~\text{when}~B=A^2/(4D).
\end{equation}
The continuity and monotone property of $g(\cdot)$ show  that
\[
H^{\rm{max}}=g(B)>0~~\text{for}~0<B<A^2/(4D)~~\text{and}~~g(B)\rightarrow 0^+~~\text{as}~B-A^2/(4D)\rightarrow 0^-.
\]
When  $0<B<A^2/(4D)$, by $(\mu,\nu)=\big\{z>0~|~H(z)>0\big\}$ we get
\[
\mu=\Big[\frac{A-\sqrt{A^2-4BD}}{2B}\Big]^{\frac{1}{\delta}}~~\text{and}~~\nu=\Big[\frac{A+\sqrt{A^2-4BD}}{2B}\Big]^{\frac{1}{\delta}}
\]
and
\[
\nu-\mu\rightarrow 0^+~\text{as}~B-A^2/(4D)\rightarrow 0^-.
\]
Moreover, a simple calculation shows that
\[
g(B)=h(z_0,B)\geqslant h\big(B^{-\frac{1}{1+\delta}},B\big)=B^{-\frac{1}{1+\delta}}\big[A-DB^{\frac{\delta}{1+\delta}}\big]-1,
\]
which implies that $g(B)\rightarrow+\infty$ as $B\rightarrow 0$.
By \eqref{5.1} and the continuity   of $g(\cdot)$,  for any  $p>0$ there exists  $B(p)\in (0, A^2/(4D))$ such that  $H^{\rm{max}}=g(B(p))=p$ and $B(p)\rightarrow  A^2/(4D)$ as $p\rightarrow0^+$.
\end{proof}

Now for small $\eta>0$, define
\[
c_\eta(\lambda)=\lambda^{-1}\Big[\int_{\mathbb R}k(x)e^{\lambda x}dx-1+f'(0)-\eta\Big] ~~\text{for}~\lambda\in(\lambda^-,0)\cup(0,\lambda^+),
\]
and
\[
c_r^*(\eta)=\inf\limits_{\lambda\in(0,\lambda^+)}\{c_\eta(\lambda)\}~\text{and}~c_l^*(\eta)=\sup\limits_{\lambda\in(\lambda^-,0)}\{c_\eta(\lambda)\}~~
\text{for}~\eta\in(0,f'(0)).
\]
It follows that $c_l^*<c_l^*(\eta)<c_r^*(\eta)<c_r^*$ and $c_l^*(\eta)\rightarrow c_l^*$, $c_r^*(\eta)\rightarrow c_r^*$ as $\eta\rightarrow0$.
Then for any small $\epsilon>0$, we can choose two   constants $\eta_1$ and $\eta_2$  in $(0,f'(0))$ such that
\begin{equation}\label{5.12}
c_r^*(\eta_1)=c_r^*-\epsilon,~~c_l^*(\eta_2)=c_l^*+\epsilon.
\end{equation}
For any $\eta\in(0,f'(0))$,  define
\begin{equation}\label{5.92}
G_\eta(c,\lambda)=c\lambda-\int_{\mathbb R}k(x)e^{\lambda x}dx+1-f'(0)+\eta~~~\text{for}~~c\in(c_l^*,c_r^*)~\text{and}~\lambda\in(\lambda^-,\lambda^+).
\end{equation}
It is easy to check that
\begin{equation}\label{5.79}
G_\eta(c,0)<0~\text{and}~\frac{\partial^2}{\partial\lambda^2}G_\eta(c,\lambda)<0.
\end{equation}
From  \eqref{1.9} and \eqref{1.10}, it follows that
\begin{equation}\label{5.80}
\lim\limits_{\lambda\rightarrow\lambda^{\pm}}G_\eta(c,\lambda)=-\infty.
\end{equation}

Consider the case   $\eta=\eta_1$.
For   $\lambda\in(0,\lambda^+)$ and  $c_1\in(c_r^*-\epsilon,c_r^*)$, we have
\[
G_{\eta_1}(c_1,\lambda)=[c_1-c_{\eta_1}(\lambda)]\lambda=[c_1-c_r^*(\eta_1)]\lambda+[c_r^*(\eta_1)-c_{\eta_1}(\lambda)]\lambda.
\]
Using the same argument as in the proof of Lemma \ref{th2.99}, we can find some constant $\lambda\in(0,\lambda^+)$ such that $c_r^*(\eta_1)=c_{\eta_1}(\lambda)$.
Therefore, for any $c_1\in(c_r^*-\epsilon,c_r^*)$ there is some constant $\lambda\in(0,\lambda^+)$ such that $G_{\eta_1}(c_1,\lambda)>0$.
By \eqref{5.79} and \eqref{5.80}, for any $c_1\in(c_r^*-\epsilon,c_r^*)$, there are three constants $\alpha^+(c_1)$, $\beta^+(c_1)$ and $\gamma^+(c_1)$  in $(0,\lambda^+)$ such that $\alpha^+(c_1)<\gamma^+(c_1)<\beta^+(c_1)$ and
\begin{equation}\label{5.13}
G_{\eta_1}(c_1,\alpha^+(c_1))=G_{\eta_1}(c_1,\beta^+(c_1))=0~\text{and}~
G_{\eta_1}(c_1,\gamma^+(c_1))>0.
\end{equation}
Moreover, we have that $G_{\eta_1}(c_1,\lambda)>0$ for $\lambda\in(\alpha^+(c_1),\beta^+(c_1))$.

Similarly,  when considering the case  $\eta=\eta_2$, for any $c_2\in(c_l^*,c_l^*+\epsilon)$ we can find three constants $\alpha^-(c_2)$, $\beta^-(c_2)$ and $\gamma^-(c_2)$ in $(\lambda^-,0)$  such that  $\beta^-(c_2)<\gamma^-(c_2)<\alpha^-(c_2)$ and
\begin{equation}\label{5.91}
G_{\eta_2}(c_2,\alpha^-(c_2))=G_{\eta_2}(c_2,\beta^-(c_2))=0~\text{and}~G_{\eta_2}(c_2,\gamma^-(c_2))>0
\end{equation}
Moreover, it follows that $G_{\eta_2}(c_2,\lambda)>0$ for $\lambda\in(\beta^-(c_2),\alpha^-(c_2))$.

The following theorem and its proof are the main results of this section.

\begin{theorem}[Spreading speeds]\label{th2.1}
Suppose that the assumptions {\rm(H)}, {\rm(K1)} and {\rm{(K2)}} hold. If $u_0(\cdot)$ satisfies  that  $0\leqslant u_0(x)\leqslant 1$ for $x\in\mathbb R$, $u_0(x_1)>0$~for some $x_1\in\mathbb R$ and
\begin{equation}
u_0(x)e^{\lambda_l^*x}\leqslant  \Gamma~\text{for}~x \leqslant -x_0,~~~u_0(x)e^{\lambda_r^*x}\leqslant  \Gamma~\text{for}~x \geqslant x_0,
\end{equation}
where $x_0$ and $\Gamma$ are two positive constants, then for any small $\epsilon>0$, there is a constant $p\in(0,1)$ such that the solution $u(t,x)$ of equation \eqref{1.1} has the following  properties:
\begin{equation}\label{1.6}
\left\{\begin{aligned}
&\lim\limits_{t\rightarrow+\infty}~\sup\limits_{x\leqslant (c_l^*-\epsilon)t}u(t, x)=0,\\
&\inf\limits_{(c_l^*+\epsilon)t \leqslant x-x_1\leqslant(c_r^*-\epsilon)t} u(t, x) \geqslant p~~~\text{for any}~~t>0,\\
&\lim\limits_{t\rightarrow+\infty}\sup\limits_{x\geqslant (c_r^*+\epsilon)t}u(t, x)=0.
\end{aligned}
\right.
\end{equation}
\end{theorem}

\begin{proof}
\emph{Step 1} (\emph{lower solution and ``forward-backward spreading'' method}).
From $u_0(x_1)>0$,  by translating the $x$-axis, we can find two positive  constants $p_1$ and $r$ such that
\begin{equation}\label{5.89}
u_0(x)\geqslant p_1~~\text{for}~~x\in[-r,r].
\end{equation}
For  small $\epsilon>0$, let $\eta_1\in\mathbb R^+$ and $\eta_2\in\mathbb R^+$ be the constants satisfying \eqref{5.12}.
By $f(u)\in C^1([0,1])$ and $f'(0)>0$,  there is a constant $p_2\in(0, p_1]$ such that
\[
f(u)\geqslant (f'(0)-\frac{\eta}{2})u~~\text{for}~u\in [0,p_2],
\]
where $\eta=\min\{\eta_1,\eta_2\}$. For any $\delta\in(0,1)$, by taking $M(\delta)=\eta p_2^{-\delta}/2$, we can get that
\begin{equation}\label{5.36}
f(u)\geqslant (f'(0)-\eta_i)u+M(\delta)u^{1+\delta}~~\text{for}~u\in [0,p_2].
\end{equation}

Now we  prove that for any  $c_1\in(c_r^*-\epsilon,c_r^*)$ and $c_2\in(c_l^*,c_l^*+\epsilon)$ there is a constant $p\in(0,1)$ such that
\[
u(\tau,X)\geqslant p~~\text{for any given}~\tau> 0,~X\in[c_2\tau,c_1\tau],
\]
Let $\kappa$ be a constant defined by
\[
\kappa=\frac{X-c_2\tau}{c_1\tau-c_2\tau}\in[0,1].
\]
In the following proof, we give the ``forward-backward spreading'' method and  divide the time  period of $[0,\tau]$ into two parts $[0,\kappa \tau]$ and $[\kappa \tau,\tau]$.

In $[0,\kappa \tau]$,  for $c_1\in(c_r^*-\epsilon,c_r^*)$ we choose the same  $\alpha^+(c_1)$, $\beta^+(c_1)$ and $\gamma^+(c_1)$ as those in \eqref{5.13}.
Construct a set of lower solutions  which spread at a speed of $c_1$  as follows
\begin{equation}\label{5.86}
\underline u_1(t,x;\xi_1)=\max\big\{0,~H_1(e^{\rho_1(-x+c_1t+\xi_1)})\big\}~~\text{for}~t\in[0,\kappa \tau],~x\in\mathbb R
\end{equation}
with
\begin{equation}\label{5.85}
\xi_1\in\left[-r+\rho_1^{-1}\ln\nu_1,~r+\rho_1^{-1}\ln\mu_1\right],
\end{equation}
where
\[
\begin{aligned}
&H_1(z)=A_1z-B_1z^{1+\delta_1}-D_1z^{1-\delta_1}~~\text{for}~z>0,\\
&\rho_1=\frac{\beta^+(c_1)+\gamma^+(c_1)}{2},~~\delta_1=\frac{\beta^+(c_1)-\gamma^+(c_1)}{\beta^+(c_1)+\gamma^+(c_1)},
\\
&(A_1)^\delta=\frac{G_{\eta_1}(c_1,\rho_1)}{M(\delta_1)},~~~D_1=\frac{A_1 G_{\eta_1}(c_1,\rho_1)}{G_{\eta_1}(c_1,\gamma^+(c_1))},~B_1\in\left(0,~A_1^2/(4D_1)\right),\\
&(\mu_1,\nu_1)\triangleq\big\{z>0~|~H_1(z)>0\big\}.
\end{aligned}
\]
Here $G_\eta(c,\lambda)$ is defined by \eqref{5.92}.
By Lemma \ref{th3.4}, we can choose $B_1\in\left(0,A_1^2/(4D_1)\right)$ close  to $A_1^2/(4D_1)$ such that
\begin{equation}\label{5.14}
H_1^{\rm{max}}\triangleq\sup\limits_{z>0}\big\{H_1(z)\big\}\leqslant p_2\leqslant p_1,~~\rho_1^{-1}\big(\ln\nu_1-\ln\mu_1\big)\leqslant r/2.
\end{equation}
Let $z_1$ be the constant in $(\mu_1,\nu_1)$ such that $H_1^{\rm{max}}=H_1(z_1)$. A simple calculation  implies that
\[
\underline u_1(0,x;\xi_1)=
\left\{
\begin{aligned}
&0 &\text{for}~x\notin\Omega_1,\\
&H_1(e^{\rho_1(-x+\xi_1)}) &\text{for}~x\in\Omega_1,
\end{aligned}
\right.
\]
where
\[
\Omega_1=(\xi_1-\rho_1^{-1}\ln \nu_1,\xi_1-\rho_1^{-1}\ln \mu_1).
\]
By \eqref{5.85} we have that $\Omega_1\subseteq(-r,r)$. From \eqref{5.89} and $H_1^{\rm{max}}\leqslant p_1$, it follows that
\[
\underline u_1(0,x;\xi_1)\leqslant u_0(x)~~\text{for}~x\in\mathbb R.
\]

Next we verify  that $\underline u_1(t,x;\xi_1)$ is a lower solution of equation \eqref{1.1}.
When $x-c_1t\notin \overline \Omega_1 $,  it is easy to check that $u_1(t,x;\xi_1)=0$ and
\[
\partial_t\underline u_1(t,x;\xi_1)- k*\underline u_1(t,x;\xi_1)+\underline u_1(t,x;\xi_1)- f(\underline u_1(t,x;\xi_1))\leqslant 0.
\]
When $x-c_1t\in  \Omega_1 $, we  have that $u_1(t,x;\xi_1)=H_1(e^{\rho_1(-x+c_1t+\xi_1)})$.
By \eqref{5.36},  some calculations show that
\[
\begin{aligned}
&\partial_t\underline u_1(t,x;\xi_1)- k*\underline u_1(t,x;\xi_1)+\underline u_1(t,x;\xi_1)- f(\underline u_1(t,x;\xi_1))\\
&\quad \quad \leqslant A_1G_{\eta_1}(c_1,\rho_1)e^{\rho_1(-x+c_1t+\xi_1)}-B_1G_{\eta_1}(c_1,\rho_1(1+\delta_1))e^{\rho_1(1+\delta_1)(-x+c_1t+\xi_1)}\\
&\quad\quad\quad -D_1G_{\eta_1}(c_1,\rho_1(1-\delta_1))e^{\rho_1(1-\delta_1)(-x+c_1t-\xi_1)}+M(\delta_1)A_1^{1+\delta}e^{\rho_1(1+\delta_1)(-x+c_1t+\xi_1)}.
\end{aligned}
\]
Recall the definitions of $\rho_1$, $\delta_1$, $A_1$ and $D_1$, then we get from  \eqref{5.92}  that
\[
\begin{aligned}
&G_{\eta_1}(c_1,\rho_1(1+\delta_1))=G_{\eta_1}(c_1,\beta^+(c_1))=0,\\
&D_1G_{\eta_1}(c_1,\rho_1(1-\delta_1))=D_1G_{\eta_1}(c_1,\gamma^+(c_1))=A_1 G_{\eta_1}(c_1,\rho_1),\\
&M(\delta_1)(A_1)^{1+\delta}=A_1G_{\eta_1}(c_1,\rho_1).
\end{aligned}
\]
It follows that
\[
\begin{aligned}
&\partial_t\underline u_1(t,x;\xi_1)- k*\underline u_1(t,x;\xi_1)+\underline u_1(t,x;\xi_1)- f(\underline u_1(t,x;\xi_1))\\
\leqslant~&A_1G_{\eta_1}(c_1,\rho_1)\big[e^{\rho_1(-x+c_1t+\xi_1)}-e^{\rho_1(1+\delta_1)(-x+c_1t+\xi_1)}-e^{\rho_1(1-\delta_1)(-x+c_1t+\xi_1)}\big]\leqslant0.
\end{aligned}
\]
Then $\underline u_1(t,x;\xi_1)$ is a lower solution of equation \eqref{1.1}.

Therefore, Lemma \ref{th3.1} implies
\[
u(t,x)\geqslant \underline u_1(t,x;\xi_1)~~\text{for}~t\in[0,\kappa \tau],~x\in\mathbb R.
\]
Define $x_1(t)=c_1t+\xi_1-\rho_1^{-1}\ln z_1$ with $t\in[0,\kappa \tau]$ and it follows that
\[
u(t,x_1(t))\geqslant \underline u_1(t,x_1(t);\xi_1)=H_1^{\text{max}} ~\text{for}~t\in[0,\kappa \tau],
\]
The  arbitrariness of the parameter $\xi_1$ in \eqref{5.85} shows that
\[
u(t,x)\geqslant H_1^{\text{max}} ~\text{for}~t\in[0,\kappa \tau],~x\in[c_1t-r/2,~c_1t+r/2].
\]
Then we have
\begin{equation}\label{5.84}
u(\kappa \tau,x)\geqslant H_1^{\text{max}} ~\text{for}~x\in[c_1\kappa \tau-r/2,~c_1\kappa \tau+r/2].
\end{equation}

In the second time period  $[\kappa \tau,\tau]$, for $c_2\in(c_l^*,c_l^*+\epsilon)$ we choose  the same $\alpha^-(c_2)$, $\beta^-(c_2)$ and $\gamma^-(c_2)$ as those in \eqref{5.91}.
Construct  another set of lower solutions  which spread at a speed of $c_2$  as follows
\begin{equation}\label{5.81}
\begin{aligned}
\underline u_2(t,x;\xi_2)&=\max\big\{0,~H_2(e^{\rho_2(-x+c_2t+\xi_2)})\big\}~\text{for}~t\in[\kappa \tau,\tau],~x\in\mathbb R\\
&=\left\{
\begin{aligned}
&0~&\text{for}~x-c_2t\notin\Omega_2,\\
&H_2(e^{\rho_2(-x+c_2t+\xi_2)})~&\text{for}~x-c_2t\in\Omega_2
\end{aligned}
\right.
\end{aligned}
\end{equation}
with
\begin{equation}\label{5.83}
\xi_2\in\left[(c_1-c_2)\kappa\tau+\rho_2^{-1}\ln\nu_2-r/2,~(c_1-c_2)\kappa\tau+\rho_2^{-1}\ln\mu_2+r/2\right],
\end{equation}
where
\[
\begin{aligned}
&\Omega_2=(~\xi_2-\rho_2^{-1}\ln \nu_2,~\xi_2-\rho_2^{-1}\ln \mu_2),\\
&H_2(z)=A_2z-B_2z^{1+\delta_2}-D_2z^{1-\delta_2}~~\text{for}~z>0,\\
&\rho_2=\frac{\beta^-(c_2)+\gamma^-(c_2)}{2},~~\delta_2=\frac{\beta^-(c_2)-\gamma^-(c_2)}{\beta^-(c_2)+\gamma^-(c_2)},
\\
&(A_2)^\delta=\frac{G_{\eta_2}(c_2,\rho_2)}{M(\delta_2)},~~~D_2=\frac{A_2 G_{\eta_2}(c_2,\rho_2)}{G_{\eta_2}(c_2,\gamma^-(c_2))},~B_2\in\left(0, A_2^2/(4D_2)\right),\\
&(\mu_2,\nu_2)\triangleq\left\{z>0~|~H_2(z)>0\right\}.
\end{aligned}
\]
Here $G_\eta(c,\lambda)$ is defined by \eqref{5.92}. By Lemma \ref{th3.4}, we can choose $B_2\in\left(0, A_2^2/(4D_2)\right)$  close  to $A_2^2/(4D_2)$ such that
\begin{equation}\label{5.15}
H_2^{\rm{max}}\triangleq\sup\limits_{z>0}\big\{H_2(z)\big\}\leqslant H_1^{\rm{max}}\leqslant p_2\leqslant p_1,~~\rho_2^{-1}\big(\ln\nu_2-\ln\mu_2\big)\leqslant r/2,
\end{equation}
Let $z_2$ be the constant in $(\mu_2,\nu_2)$ such that $H_2^{\rm{max}}=H_2(z_2)$. At time $t=\kappa\tau$, we have
\[
\underline u_2(\kappa \tau,x;\xi_2)
=\left\{
\begin{aligned}
&0~&\text{for}~x\notin\Omega_2+c_2\kappa \tau,\\
&H_2(e^{\rho_2(-x+c_2\kappa \tau+\xi_2)})~&\text{for}~x\in\Omega_2+c_2\kappa \tau,
\end{aligned}
\right.
\]
where
\[
\Omega_2+c_2\kappa \tau\triangleq (~\xi_2-\rho_2^{-1}\ln \nu_2+c_2\kappa \tau,~\xi_2-\rho_2^{-1}\ln \mu_2+c_2\kappa \tau)\subseteq(c_1\kappa\tau-r/2,c_1\kappa\tau+r/2).
\]
Then it follows from \eqref{5.84} that $u(\kappa \tau,x)\geqslant\underline u_2(\kappa \tau,x;\xi_2)$ for $x\in\mathbb R$ and any $\xi_2$ satisfying \eqref{5.83}.
Similarly to the case of $\underline u_1(t,x;\xi_1)$, it can be verified  that $\underline u_2(t,x;\xi_2)$ is a lower solution of equation \eqref{1.1} in time $[\kappa \tau,\tau]$. Then  for any $\xi_2$ satisfying \eqref{5.83}, by Lemma \ref{th3.1} we have that
\[
u(t,x)\geqslant\underline u_2(t,x;\xi_2)~\text{for}~t\in[\kappa \tau,\tau],~x\in\mathbb R.
\]

For $t\in[\kappa \tau,\tau]$, we define $x_2(t)=c_2t+\xi_2-\rho_2^{-1}\ln z_2$  and  it follows that
\[
u(t,x_2(t))\geqslant\underline u_2(t,x_2(t);\xi_2)=H(z_2)=H_2^{\text{max}}.
\]
Since $\rho_2^{-1}\ln \nu_2-r/2\leqslant\rho_2^{-1}\ln z_2\leqslant\rho_2^{-1}\ln\mu_2+r/2$, we can choose $\xi_2$ satisfying \eqref{5.83} and
\[
\xi_2=(c_1-c_2)\kappa \tau+\rho_2^{-1}\ln z_2,
\]
It follows that $x_2(\tau)=c_1\kappa\tau+c_2(1-\kappa)\tau=X$.  By taking $p=H_2^{\text{max}}$, we have that
\[
u(\tau,X)\geqslant p~~\text{for any}~\tau>0,~X\in[c_2\tau,c_1\tau].
\]
Therefore,   for any   small $\epsilon>0$ there is a constant $p\in(0,1)$ such that
\[
\inf\limits_{(c_l^*+\epsilon)t \leqslant x\leqslant(c_r^*-\epsilon)t} u(t, x) \geqslant p~~~\text{for any}~~t>0.
\]

\emph{Step 2} (\emph{upper solution}). Now we begin to prove that
\begin{equation}\label{5.98}
\lim\limits_{t\rightarrow+\infty}~\sup\limits_{x\leqslant (c_l^*-\epsilon)t}u(t, x)=0~~\text{and}~~
\lim\limits_{t\rightarrow+\infty}\sup\limits_{x\geqslant (c_r^*+\epsilon)t}u(t, x)=0.
\end{equation}
Construct an upper solution as follows
\[
\bar u(t,x)=\min\left\{1,~\Gamma_0e^{\lambda_r^*(-x+c_r^*t)},~\Gamma_0 e^{\lambda_l^*(-x+c_l^*t)}\right\},
\]
where the constant $\Gamma_0\geqslant \max\{1,\Gamma\}$  is large enough   such that $\bar u(0,x)\geqslant u_0(x)$.

Next we verify  that $\bar u(t,x)$ is an upper solution of equation \eqref{1.1}.
Define
\[
G(c,\lambda)=c\lambda-\int_{\mathbb R}k(x)e^{\lambda x}dx+1-f'(0)~~\text{for}~c\in\mathbb R ,~\lambda\in (\lambda^-,\lambda^+).
\]
Then it follows from \eqref{2.2} and \eqref{2.3}  that $G(c_r^*,\lambda_r^*)=G(c_l^*,\lambda_l^*)=0$. By a simple calculation, if $x\leqslant c^*_lt+(\lambda_l^*)^{-1}\ln \Gamma_0$, then $\bar u(t,x)=\Gamma_0e^{\lambda_l^*(-x+c_l^*t)}$ and it follows from (H) that
\[
\bar u_t(t,x)- k* \bar u(t,x)+\bar u(t,x) - f(\bar u(t,x))\geqslant G(c_l^*,\lambda_l^*)\Gamma_0e^{\lambda_l^*(-x+c_l^*t)}=0.
\]
If $x\geqslant c^*_rt+(\lambda_r^*)^{-1}\ln \Gamma_0$, then $\bar u(t,x)=\Gamma_0 e^{\lambda_r^*(-x+c_r^*t)}$ and we get from (H) that
\[
\bar u_t(t,x)- k* \bar u(t,x)+\bar u(t,x) - f(\bar u(t,x))\geqslant G(c_r^*,\lambda_r^*)\Gamma_0 e^{\lambda_r^*(-x+c_r^*t)}=0.
\]
If  $c^*_lt+(\lambda_l^*)^{-1}\ln \Gamma_0\leqslant x\leqslant c^*_rt+(\lambda_r^*)^{-1}\ln \Gamma_0$, then $\bar u(t,x)=1$ and
\[
\bar u_t(t,x)- k* \bar u(t,x)+\bar u(t,x) - f(\bar u(t,x))\geqslant0.
\]
Therefore,  we get that $\bar u(t,x)$ is an upper solution of equation \eqref{1.1}. Lemma \ref{th3.1} implies that $u(t,x)\leqslant \bar u(t,x)$ for $t\geqslant0$, $x\in\mathbb R$. It follows that
\[
\begin{aligned}
&\sup\limits_{x\leqslant(c_l^*-\epsilon)t}u(t,x)\leqslant\sup\limits_{x\leqslant(c_l^*-\epsilon)t}\bar u(t,x)\leqslant \Gamma_0e^{\lambda_l^*\epsilon t},\\
&\sup\limits_{x\geqslant (c_r^*+\epsilon)t}u(t, x)\leqslant \sup\limits_{x\geqslant (c_r^*+\epsilon)t}\bar u(t, x)\leqslant  \Gamma_0e^{-\lambda_r^*\epsilon t},
\end{aligned}
\]
which means that \eqref{5.98} holds.
\end{proof}

By Theorem 6.2 in the  classic spreading speed
theory \cite{Wei1982}, we can get from the second inequality of \eqref{1.6} that for any small $\epsilon>0$,
\[
\inf\limits_{(c_l^*+\epsilon)t \leqslant x\leqslant (c_r^*-\epsilon)t}u(t, x)= 1~~\text{as}~t\rightarrow+\infty.\\
\]
Then combining with the other two inequalities of \eqref{1.6}, we see that   $u(t,x)$ satisfies the propagation property \eqref{1.2}.

By  the new lower solutions and the ``forward-backward spreading'' method  above, we  give a corollary which shows that if $u_0(x_1)>0$  for some $x_1\in\mathbb R$, then the property $u>0$ will spread over an expanding spatial  region.

\begin{corollary}\label{th2.5}
Suppose that {\rm(H)}, {\rm(K1)} and {\rm(K2)} hold. For any small $\epsilon>0$ and small $p>0$, there is a constant  $r_\epsilon(p)>0$ such that if
\[
u_0(x)\geqslant p,~x\in[x_1-r_\epsilon(p),x_1+r_\epsilon(p)]~~\text{for some}~x_1\in\mathbb R,
\]
then  the solution $u(t,x)$ of equation \eqref{1.1} satisfies that
\[
u(t,x)\geqslant p~~\text{for}~t>0,~x\in[x_1+(c_l^*+\epsilon)t,~x_1+(c_r^*-\epsilon)t].
\]
Moreover, for any small $\epsilon>0$, we have that $r_\epsilon(p)\rightarrow0$ as $p\rightarrow0$.
\end{corollary}

\begin{proof}
We use the same notations as those in the proof of Theorem \ref{th2.1}. By translating the $x$-axis, we suppose that $x_1=0$.
From Lemma \ref{th3.4}, for any  $p\in(0,p_2]$, there are $B_1(p)\in\left(0,~A_1^2/(4D_1)\right)$ and $B_2(p)\in\left(0, A_2^2/(4D_2)\right)$ satisfying  that $H_1^{\text{max}}=H_2^{\text{max}}=p$.
We define  $r(p)=2(r_1(p)+r_2(p))$, where $r_i(p)$ is the length of the set $\{x\in\mathbb R~|~H_i(e^{-\rho_ix})>0\}$.
We suppose that
\[
u_0(x)\geqslant p~~\text{for}~x\in[-r(p),r(p)].
\]

Define the lower solutions $\underline u_1(t,x;\xi_1)$ and $\underline u_2(t,x;\xi_2)$ by \eqref{5.86} and \eqref{5.81}, respectively, where
\[
\begin{aligned}
&\xi_1\in\left[-r(p)+\rho_1^{-1}\ln\nu_1,~r(p)+\rho_1^{-1}\ln\mu_1\right],\\
&\xi_2\in\left[(c_1-c_2)\kappa\tau+\rho_2^{-1}\ln\nu_2-r_2(p),~(c_1-c_2)\kappa\tau+\rho_2^{-1}\ln\mu_2+r_2(p)\right].
\end{aligned}
\]
It follows that
\[
\Omega_1\subseteq (-r(p),r(p)),~\Omega_2+c_2\kappa \tau\subseteq(c_1\kappa\tau-r_2(p),c_1\kappa\tau+r_2(p)).
\]
Then by the same method as the proof of Theorem \ref{th2.1}, we can prove Corollary \ref{th2.5}. Moreover, as $p\rightarrow0^+$, it follows from Lemma \ref{th3.4} that $B_i(p)\rightarrow A_i^2/(4D_i)$, which implies that $r_i(p)\rightarrow0$.
We can see that $r_i(p)$ is dependent on $c_i$, since $H_i$ and $\rho_i$ are   dependent on $c_i$. Therefore, $r_i(p)$ is also dependent on $\epsilon$.
\end{proof}

\begin{remark}\rm\label{re3.5}
When  considering a reaction-diffusion equation or when the kernel in equation \eqref{1.1} is symmetric,  we point out that the new lower solutions \eqref{5.86} and \eqref{5.81} remain  available.
However, it is not necessary to apply the ``forward-backward spreading'' method, since we can use Theorem  \ref{th2.4} instead (more details are found in proof of Theorem \ref{th2.3}).  Then the conclusion  in Corollary \ref{th2.5}  also holds in the reaction-diffusion equation \eqref{1.4}.
\end{remark}

\section{Spreading speeds for exponentially decaying initial data}

In this section we establish the relationship between spreading speed and exponentially decaying initial data. First we state the  weak ``hair-trigger'' effect in nonlocal dispersal equations (see e.g. \cite{AW1978,FT2018,Alfaro2017}).

\begin{lemma}[Weak ``hair-trigger'' effect]\label{th3.3}
Suppose that (H) holds and $k(\cdot)$ is a symmetric kernel satisfying (K1).
Let $u(t,x)$ be the solution of equation \eqref{1.1} with initial data $u_0(x)$.
If there are two constants $x_0\in \mathbb R$ and  $\omega_0\in (0,1)$ such that
\[
u_0(x)\geqslant\omega_0~~\text{for}~~x\in B_1(x_0),
\]
then for any $\omega\in (0,1)$, there exists $T_{\omega_0}^\omega\geqslant0$ (independent of $x_0$) such that
\[
u(t,x)\geqslant\omega~~\text{for}~~x\in B_1(x_0),~t\geqslant T_{\omega_0}^\omega,
\]
where $B_1(x_0)\triangleq\big\{x\in\mathbb R~\big|~|x-x_0|\leqslant1\big\}$.
\end{lemma}

The following theorem is the main result of this section.

\begin{theorem}\label{th2.3}
Suppose {\rm(H)} holds and $k(\cdot)$  is a symmetric kernel which is decreasing on $\mathbb R^+$ and  satisfies {\rm(K1)}. Denote $c^*\triangleq c_r^*=-c_l^*$ and $\lambda^*\triangleq \lambda_r^*=-\lambda_l^*$. If $f\in C^{1+\delta_0}\big([0,p_0]\big)$ for some  $p_0, \delta_0\in(0,1)$  and $u_0(\cdot)$ satisfies  that
\[
0< u_0(x)\leqslant 1~\text{for}~x\in\mathbb R,~u_0(x)\sim O(e^{-\lambda |x|})~~\text{as}~|x|\rightarrow +\infty~\text{with}~ \lambda\in(0,\lambda^*),
\]
then for any $\epsilon\in(0,c(\lambda))$, the solution $u(t,x)$ of equation \eqref{1.1} has the following properties
\[
\left\{\begin{aligned}
&\inf\limits_{|x|\leqslant (c(\lambda)-\epsilon)t}u(t, x)\rightarrow1,\\
&\sup\limits_{|x|\geqslant (c(\lambda)+\epsilon)t}u(t, x)\rightarrow0
\end{aligned}
\right.~\text{as}~t\rightarrow+\infty,
\]
where $c(\lambda)$ is defined by \eqref{2.1}. Moreover, we have that $c'(\lambda)<0$ for $\lambda\in(0,\lambda^*)$.
\end{theorem}

\begin{proof}
From the proof of Lemma \ref{th2.99}, we have that
\[
c'(\lambda)<0~\text{for}~\lambda\in(0,\lambda^*)~~\text{and}~c'(\lambda)>0~\text{for}~\lambda\in(\lambda^*,\lambda^+).
\]
Since $c(\lambda)\rightarrow+\infty$ as  $\lambda\rightarrow\lambda^+$,
for any $\lambda\in(0,\lambda^*)$ there is a unique constant $\delta_\lambda>0$ such that
\[
c(\lambda)=c(\lambda(1+\delta_\lambda))~\text{and}~c(s)<c(\lambda)~\text{for}~s\in(\lambda,\lambda(1+\delta_\lambda)).
\]
Define
\[
G(c,\lambda)=c\lambda-\int_{\mathbb R}k(x)e^{\lambda x}dx+1-f'(0)~~\text{for}~c\geqslant c^*,~\lambda\in (\lambda^-,\lambda^+).
\]
For any $\lambda\in(0,\lambda^*)$, it follows from \eqref{2.1} that
\[
G(c(\lambda),\lambda)=G(c(\lambda),\lambda(1+\delta_\lambda))=0
\]
and
\[
G(c(\lambda),s)>G(c(s),s)=0~\text{for}~s\in(\lambda,\lambda(1+\delta_\lambda)).
\]

Now we prove that for any $\epsilon\in(0,c(\lambda))$,
\begin{equation}\label{5.82}
\inf\limits_{|x|\leqslant (c(\lambda)-\epsilon)t}u(t, x)\rightarrow1~\text{as}~t\rightarrow +\infty.
\end{equation}
By the assumptions of $u_0$ in Theorem \ref{th2.3}, there is a function $v_0(\cdot)\in C(\mathbb R) $ which  is symmetric   and  decreasing on  $\mathbb R^+$ and satisfies that
\begin{equation}\label{5.16}
u_0(x)\geqslant v_0(x)=\left\{
\begin{aligned}
&\gamma e^{-\lambda|x|}, &&|x|\geqslant y_0,\\
&p_1\triangleq\gamma e^{-\lambda y_0}, &&|x|\leqslant y_0,\\
\end{aligned}
\right.
\end{equation}
where $\gamma$ and $y_0$ are two positive constants. Let $v(t,x)$ be the solution of equation \eqref{1.1} with the initial condition $v(0,x)=v_0(x)$.
From Lemma \ref{th3.1} it follows that
\begin{equation}\label{5.33}
u(t,x)\geqslant v(t,x)~~\text{for}~t\geqslant 0,~x\in \mathbb R.
\end{equation}
Theorem  \ref{th2.4} shows that $v(t,\cdot)$ is also symmetric  and decreasing on  $\mathbb R^+$ for any $t>0$.
We denote that $p\triangleq\min\{p_0,p_1\}$ and $\delta\triangleq\min\{\delta_0,\delta_\lambda/2\}$, then $G(c(\lambda),\lambda(1+\delta))>0$.
By $f(\cdot)\in C^{1+\delta_0}\left([0,p_0]\right)$  we can find some constant $M>0$ such that
\begin{equation}\label{5.34}
f(u)\geqslant f'(0)u-Mu^{1+\delta}~~\text{for}~u\in(0,p].
\end{equation}

Construct a lower solution as follows
\[
\underline u(t,x)=\max\left\{0,~g\left(\gamma e^{\lambda(-x+c(\lambda)t)}\right)\right\}~~\text{for}~t\geqslant0,~x\in\mathbb R,
\]
where $g(z)=z-Lz^{1+\delta}$ for $z>0$ and
\begin{equation}\label{5.35}
L\geqslant\max\left\{p^{-\delta},~\gamma^{-\delta}e^{\lambda\delta},~M/G(c(\lambda),\lambda(1+\delta))\right\}.
\end{equation}
Let $z_0$ be the constant satisfying  $z_0^{\delta}=L^{-1}(1+\delta)^{-1}$, then $\omega_0\triangleq g(z_0)\geqslant g(z)$ for all $z>0$ and
\[
\underline u(t,x)\leqslant\omega_0=L^{-\frac{1}{\delta}}\delta(1+\delta)^{-\frac{1+\delta}{\delta}}\leqslant p~~\text{for}~t\geqslant0,~x\in\mathbb R.
\]
From \eqref{5.16} it follows that $v_0(x)\geqslant\underline u(0,x)$ for $x\in\mathbb R$.
Now we verify that $\underline u(t,x)$ is a lower solution of equation \eqref{1.1}.
If $x< c(\lambda)t+\lambda^{-1}(\ln \gamma+\delta^{-1}\ln L)$, it is easy to check that  $\underline u(t,x)=0$ and
\[
\underline u_t(t,x)-k*\underline u(t,x) +\underline u(t,x)- f(\underline u(t,x))\leqslant 0.
\]
If $x\geqslant c(\lambda)t+\lambda^{-1}(\ln \gamma+\delta^{-1}\ln L)$, then $\underline u(t,x)=g\left(\gamma e^{\lambda(-x+c(\lambda)t)}\right)$.  From \eqref{5.34}  it follows that
\[
\begin{aligned}
&\underline u_t(t,x)- k*\underline u(t,x)+\underline u(t,x)- f(\underline u(t,x))\\
\leqslant&~G(c(\lambda),\lambda)\gamma e^{\lambda(-x+c(\lambda)t)}- \big[G(c(\lambda),\lambda(1+\delta))L-M \big]\gamma^{1+\delta} e^{\lambda(1+\delta)(-x+c(\lambda)t)}
\end{aligned}
\]
By $G(c(\lambda),\lambda)=0$ and $L\geqslant M/G(c(\lambda),\lambda(1+\delta))$, we get that $\underline u(t,x)$ is a lower solution.

Lemma \ref{th3.1} implies that
\[
v(t,x)\geqslant\underline u(t,x)~~\text{for}~t\geqslant0,~x\in\mathbb R.
\]
Let $x_0(t)=c(\lambda)t+\lambda^{-1}(\ln\gamma-\ln z_0)\geqslant 1$ with $t\geqslant 0$ and we have that
\[
v(t,x_0(t))\geqslant\underline u(t,x_0(t))=g(z_0)=\omega_0~\text{for}~t\geqslant0.
\]
The symmetry and monotone property of $v(t,\cdot)$  yield that
\[
v(t,x)\geqslant \omega_0~~\text{for}~~t\geqslant 0,~|x|\leqslant x_0(t).
\]
For any  $\omega\in(0,1)$, let $T_{\omega_0}^\omega$ be the  positive constant defined in Lemma \ref{th3.3} and we have
\[
v(t+T_{\omega_0}^\omega,x)\geqslant \omega~~\text{for}~t\geqslant 0,~|x|\leqslant x_0(t),
\]
which implies that
\[
\inf\limits_{|x|\leqslant x_0(t)-c(\lambda)T_{\omega_0}^\omega}v(t,x)\geqslant \omega~~\text{for}~t\geqslant T_{\omega_0}^\omega.
\]
For $\epsilon\in(0,c(\lambda))$, there is a constant $T\geqslant T_{\omega_0}^\omega$ (dependent on $\epsilon$ and $\omega$) such that
\[
\epsilon T\geqslant c(\lambda)T_{\omega_0}^\omega-\lambda^{-1}(\ln\gamma-\ln z_0).
\]
Then we have that $x_0(t)-c(\lambda)T_{\omega_0}^\omega\geqslant(c(\lambda)-\epsilon)t$ and
\[
\inf\limits_{|x|\leqslant (c(\lambda)-\epsilon)t}u(t,x)\geqslant\inf\limits_{|x|\leqslant (c(\lambda)-\epsilon)t}v(t,x)\geqslant \omega~~\text{for}~t\geqslant T,
\]
which completes the proof of \eqref{5.82}.

Finally, it suffices to check that for any $\epsilon>0$,
\begin{equation}\label{5.17}
\sup\limits_{|x|\geqslant (c(\lambda)+\epsilon)t}u(t, x)\rightarrow0~~\text{as}~t\rightarrow +\infty.
\end{equation}
Construct an upper solution as follows
\[
\bar u(t,x)=\min\left\{1,~\Gamma e^{\lambda (-|x|+c(\lambda)t )}\right\}~~\text{for}~t\geqslant0,~x\in\mathbb R.
\]
By the same method  as the step 2 of the proof of Theorem \ref{th2.1}, we  can get \eqref{5.17}.
\end{proof}

Combining   Theorems  \ref{th2.1} and \ref{th2.3}, when $k$ is symmetric, we obtain the relationship between spreading speed and initial data that  decays
exponentially or faster.
If $u_0(x)\sim O(e^{-\lambda |x|})$ as $|x|\rightarrow+\infty$,  then for $\lambda\in[\lambda^*,+\infty)$ the spreading speed equals to $c^*$ and for $\lambda\in(0,\lambda^*)$ the spreading speed $c(\lambda)$ decreases strictly along with the increase of  $\lambda$. Moreover, we have that $c^*=c(\lambda^*)$.
This relationship shows that the  nonlocal dispersal equation with symmetric kernel shares the same property of spreading speed as  the corresponding reaction-diffusion equation.

\section{Case studies}

In this section we show how to calculate  $E(k)$ and apply  Theorem \ref{th1.2} to two examples of dispersal kernels. For more applications to complex systems, refer to our paper \cite{XLR2019}.

\subsection{Normal distribution}
Assume that the dispersal kernel $k$ satisfies
\[
k(x)=\frac{1}{\sqrt{2\pi\sigma}}\exp\left(-\frac{(x-\alpha)^2}{2\sigma}\right),
\]
where $\alpha\in\mathbb R$ is the expectation and $\sigma>0$ is the variance. Define a  constant
\[r=\alpha/\sqrt{2\sigma}.\]
Then some calculations yield that $\text{sign}(r)=\text{sign}(J(k))$ and
\[
\begin{aligned}
E(k)&=\text{sign}(r)\left[1-\inf\limits_{\lambda\in\mathbb R}\left\{\exp\left(\alpha\lambda+\frac{\sigma}{2}\lambda^2\right)\right\}\right]\\
&=\text{sign}(r)\left(1-\exp\left(-r^2\right)\right).
\end{aligned}
\]
The following result is a straightforward consequence of Theorem \ref{th1.2} and we omit its proof.
\begin{corollary}\label{cor1.8}
When $f'(0)\geqslant1$, it holds that $c_l^*<0<c_r^*$  and when $f'(0)<1$,
there exists a constant $r^*>0$ such that

{\rm(i)} if $r>r^*$, then $0<c_l^*<c_r^*${\rm;}

{\rm(ii)} if $r=r^*$, then $0=c_l^*<c_r^*${\rm;}

{\rm(iii)} if $-r^*<r<r^*$, then $c_l^*<0<c_r^*${\rm;}

{\rm(iv)} if $r=-r^*$, then $c_l^*<c_r^*=0${\rm;}

{\rm(v)} if $r<-r^*$, then $c_l^*<c_r^*<0$.
\end{corollary}
\begin{remark}\rm
Since the  dispersal coefficient in equation \eqref{1.1} is $1$, the condition $f'(0)>1$ implies that the reaction term plays a more important role than the dispersal term; on the other hand,  the condition $f'(0)<1$ means that the dispersal term is more important. In the latter case, we show that the asymmetry level of dispersal  determines  the propagation directions.
\end{remark}

\subsection{Uniform distribution.} Suppose that  the kernel $k$  is given by
\[
k(x)=\left\{
\begin{aligned}
&\frac{1}{a-b},&\text{for}~x\in[b,a],\\
&0,&\text{for}~x\notin[b,a],
\end{aligned}
\right.
\]
where   $a\in\mathbb R^+$ and $b\in\mathbb R^-$ stand for the farthest distances
of  organism movements  during a unit time period along and against $x$-axis, respectively.
The  average moving speed is $\int k(x)xdx=(a+b)/2$. Some calculations yield that
\[
E(k)=\text{sign}(a+b)\big[1-\inf\limits_{\lambda\neq0}\left\{h(\lambda)\right\}\big],
\]
where $h(\lambda)=(e^{a\lambda}-e^{b\lambda})/(a\lambda-b\lambda)$ with $\lambda\neq0$.  Next, we define an auxiliary function  and give its  property in the following lemma.
\begin{lemma}\label{lem1.10} Define $\omega(x)=(x-1)e^x$.
Then there is a unique continuous function $z(\cdot)$ from $(0,+\infty)$ to $(-\infty,1)$ with $z(\cdot)\not\equiv0$ such that $\omega(z(\theta))=\omega(-\theta z(\theta))$ for any $\theta>0$. Moreover,  the function $z(\cdot) $ is increasing on $(0,+\infty)$.
\end{lemma}
\begin{proof}
For any $\theta>0$, define
\[
\bar \omega(\theta,x)=\omega(x)-\omega(-\theta x)=(x-1)e^x+(\theta x+1)e^{-\theta x}~~\text{for}~\theta\in(0,+\infty),~x\in\mathbb R.
\]
It follows that  $\frac{\partial}{\partial x}\bar\omega(\theta,x)=xe^x-\theta^2xe^{-\theta x}$ for $x\in\mathbb R$.
Denote  $x_1=0$ and $x_2(\theta)=2(1+\theta)^{-1}\ln \theta$, then
$\frac{\partial}{\partial x}\bar\omega(\theta,x_1)=\frac{\partial}{\partial x}\bar\omega(\theta,x_2(\theta))=0$.
Some calculations yield that
\begin{equation}\label{6.1}
\bar \omega(\theta,0)=0,~~\bar \omega(\theta,1)>0,~~\bar \omega(\theta,-1/\theta)<0,
\end{equation}
and
\[
\bar\omega\left(\theta,1-1/\theta\right)=e^{1-\theta}(\theta^2-e^{\theta-1/\theta})/{\theta}.
\]
Notice that the function $\theta \mapsto \theta-1/\theta-2\ln \theta$ is strictly increasing on $(0,+\infty)$ and it equals 0 when $\theta=1$. Then we have that
\begin{equation}\label{6.2}
\bar\omega(\theta,1-1/\theta)<0~\text{for}~\theta>1,~\bar\omega(\theta,1-1/\theta)>0~\text{for}~0<\theta<1.
\end{equation}
If $\theta>1$, then $x_1<x_2(\theta)$ and
\begin{equation}\label{6.3}
\frac{\partial}{\partial x}\bar\omega(\theta,x)<0~\text{for}~x\in(x_1,x_2(\theta)),~ \frac{\partial}{\partial x}\bar\omega(\theta,x)>0~\text{for}~x\in\mathbb R \backslash[x_1,x_2(\theta)].
\end{equation}
By \eqref{6.1} and \eqref{6.2}, for any $\theta>1$ there is a unique $z(\theta)\in(1-1/\theta,1)$ such that $\bar\omega(\theta,z(\theta))=0$.
On the other hand, if $0<\theta<1$ then $x_1>x_2(\theta)$ and
\[
\frac{\partial}{\partial x}\bar\omega(\theta,x)<0~\text{for}~x\in(x_2(\theta),x_1),~\frac{\partial}{\partial x}\bar\omega(\theta,x)>0~\text{for}~x\in\mathbb R \backslash[x_2(\theta),x_1].
\]
For any $\theta\in(0,1)$, we can find  a unique $z(\theta)\in(-1/\theta,1-1/\theta)$ such that $\bar\omega(\theta,z(\theta))=0$.
In addition, when $\theta=1$ we define $z(\theta)=0$.  Finally, we show   that
\begin{equation}\label{6.4}
z(1)=0,~z(\theta)\in(1-1/\theta,1)~\text{for}~\theta>1,~z(\theta)\in(-1/\theta,1-1/\theta)~\text{for}~0<\theta<1.
\end{equation}

Now  we prove that  $z(\cdot)$ is continuous on $(0,+\infty)$. Indeed, it suffices to show that \[
\lim\limits_{\theta\rightarrow1^+}z(\theta)=\lim\limits_{\theta\rightarrow1^-}z(\theta)=0.
\]
Notice that
\[
\bar \omega(\theta,z(\theta))-\bar\omega\left(\theta,1-1/\theta\right)=\int_{1-1/\theta}^{z(\theta)}\frac{\partial}{\partial x}\bar\omega(\theta,x)dx,
\]
which means that
\[
-e^{1-\theta}(\theta^2-e^{\theta-1/\theta})/{\theta}=\int_{1-1/\theta}^{z(\theta)}xe^x-\theta^2xe^{-\theta x}dx.
\]
Let  $\theta\rightarrow1^+$ or $1^-$, then
\[
\lim\limits_{\theta\rightarrow1^+}\int_{0}^{z(\theta)}xe^x-xe^{-x}dx=\lim\limits_{\theta\rightarrow1^-}\int_{0}^{z(\theta)}xe^x-xe^{-x}dx=0.
\]
It follows that $\lim_{\theta\rightarrow1^+}z(\theta)=\lim_{\theta\rightarrow1^-}z(\theta)=0$.
Therefore, $z(\cdot)$ is continuous on $(0,+\infty)$.

Next, we prove that $z(\cdot) $ is increasing on $(0,+\infty)$. Consider the function $\bar \omega(\theta,x)$ with $(\theta,x)\in(1,+\infty)\times(0,+\infty)$. For any fixed $\theta_0>1$, it holds that $\bar \omega(\theta_0,z(\theta_0))=0$ and $\frac{\partial}{\partial x}\bar\omega(\theta_0,z(\theta_0))>0$ by \eqref{6.3}. Then implicit function theorem implies  that
$z(\cdot)$ has a continuous derivative at $\theta_0$ and
\[
z'(\theta)=-\frac{\partial \bar\omega(\theta,z(\theta))}{\partial \theta}\left/\frac{\partial \bar\omega(\theta,z(\theta))}{\partial x}\right.
=\frac{\theta z^2(\theta)e^{-\theta z(\theta)}}{z(\theta)e^{z(\theta)}-\theta^2z(\theta)e^{-\theta z(\theta)}}~~\text{for}~ \theta>1.\\
\]
From $\omega(z(\theta))=\omega(-\theta z(\theta))$ it follows that
\begin{equation}\label{6.5}
z'(\theta)=\frac{z(\theta)(z(\theta)-1)}{(\theta+1)[1-1/\theta-z(\theta)]}~~\text{for}~ \theta>1.
\end{equation}
When $\theta>1$,  by $z(\theta)\in(1-1/\theta,1)$, we have that $z'(\theta)>0$. Similarly, we can prove that   $z'(\cdot)$ is continuous on $(0,1)$ and
\[
z'(\theta)=\frac{z(\theta)(z(\theta)-1)}{(\theta+1)[1-1/\theta-z(\theta)]}~\text{ for }~\theta\in(0,1).
\]
Then for $\theta\in(0,1)$, by $z(\theta)\in(-1/\theta,1-1/\theta)$ we obtain that $z'(\theta)>0$.
Therefore, we prove that $z(\cdot)$ is increasing on $(0,+\infty)$. This completes the proof.
\end{proof}

Define a constant
\[\theta\triangleq-a/b\in(0,+\infty).\]
From $h'(\lambda)=(\omega(a\lambda)-\omega(b\lambda))/(a\lambda^2-b\lambda^2)$, it follows that $h'(z(\theta)/b)=0$. Then by $\omega(z(\theta))=\omega(-\theta z(\theta))$, we have that $h(z(\theta)/b)=e^{z(\theta)}/(1+\theta z(\theta))$ and
\[
E(k)=\text{sign}(\theta-1)\Big[1-\frac{e^{z(\theta)}}{1+\theta z(\theta)}\Big].
\]
Denote
\[
r\triangleq(\theta-1)/(\theta+1)=(a+b)/(a-b)\in(-1,1).
\]
\begin{corollary}\label{cor1.11}
All results in Corollary \ref{cor1.8} hold for the uniform distribution case.
\end{corollary}

\begin{proof} It  suffices to prove the results in the case   $0<f'(0)<1$, since $-1<E(k)<1$. Now we only consider   the case   $r\geqslant0$, namely $\theta\geqslant1$ (otherwise   consider the new spatial variable $y=-x$).
Denote
\[
q(\theta)=1-\frac{e^{z(\theta)}}{1+\theta z(\theta)}.
\]
For $\theta>1$, it follows that
\[
q'(\theta)=\frac{e^{z(\theta)}(\theta-\theta z(\theta)-1)}{[1+\theta z(\theta)]^2}z'(\theta)+\frac{e^{z(\theta)}z(\theta)}{[1+\theta z(\theta)]^2}.
\]
From \eqref{6.5} we get that
\[
q'(\theta)=\frac{e^{z(\theta)}z(\theta)[\theta z(\theta)-\theta+1]}{[1+\theta z(\theta)]^2}~\text{for}~\theta>1.
\]
Then   \eqref{6.4} implies  that  $q'(\theta)>0$ for $\theta>1$, which means that $q(\cdot)$ is strictly increasing on $[1,+\infty)$. Moreover, since $z(\theta)\rightarrow1$ as $\theta\rightarrow+\infty$, we have that
\[
q(1)=0~\text{and}~q(\theta)\rightarrow1~\text{as}~\theta\rightarrow+\infty.
\]
Therefore, when $f'(0)\in(0,1)$, there exists a unique constant $\theta^*>1$ such that $q(\theta^*)=f'(0)$.
Denote $r^*=(\theta^*-1)/(\theta^*+1)$.
Finally, by Theorem \ref{th1.2}, the  monotone property of $q$ completes the proof.
\end{proof}

\section*{Acknowledgments}
\noindent

The authors would like to thank Prof. Chris Cosner (University of Miami), Prof. Jian-Wen Sun (Lanzhou University), Dr. Ru Hou (Peking University) and Dr. Teng-Long Cui (Lanzhou University)  for their helpful comments. Research of W.-B. Xu was partially supported by China Postdoctoral Science Foundation (2019M660047); research of W.-T. Li was partially supported by NSF of China (11731005, 11671180); and research of S. Ruan was partially supported by National Science Foundation (DMS-1853622).

\end{document}